\documentclass[10pt,reqno, final]{amsart}
\usepackage[T1]{fontenc}	
\usepackage[utf8]{inputenc}	
\usepackage[english]{babel}
\usepackage{amsmath}
\usepackage{amssymb}
\usepackage{amsthm}
\usepackage{color}

\theoremstyle{plain}
 
 \newtheorem{thm}{Theorem}[section]
 \newtheorem{cor}[thm]{Corollary}
 
 \newtheorem{lem}[thm]{Lemma}
 \newtheorem{prop}[thm]{Proposition}

\theoremstyle{definition}
	\newtheorem{defn}[thm]{Definition}
	\newtheorem{ex}[thm]{Example}

\theoremstyle{remark}
    \newtheorem{rem}[thm]{Remark}
        		
	\newtheorem{note}[thm]{Notation}

\numberwithin{equation}{section}	
\usepackage{mathtools}	% per i simboli := e =:
\usepackage{mathrsfs}	% per il math script
\usepackage{eucal}	% per il mathcal + figo
\usepackage{braket}	% per il comando \Set: definizione  automatica di 
\usepackage[all,pdf]{xy}% per i diagrammi commutativi

\usepackage{mparhack}	% fix bug for \marginpar
\usepackage{relsize}	% dimensioni testo
\usepackage{a4wide}	% aumenta i margini di pagina
\usepackage{booktabs}	% per fare bene le tabelle
\usepackage{multirow}	% per celle su pi righe nelle tabelle
\usepackage{caption}	% per le didascalie di tabelle e figure
\usepackage{rotating}
\usepackage{subfig}

\usepackage{varioref}	% per indicazione di pagina nei riferimenti 
\usepackage{footmisc}

\usepackage{graphicx}	% per le figure
\usepackage{epsfig}
\usepackage{enumerate}	% per impostare facilmente gli elenchi numerati
\usepackage[colorlinks=true, linkcolor=black, 
citecolor=black,urlcolor=blue]{hyperref}		
\mathtoolsset{showonlyrefs}

\newcommand{\N}{\mathbb{N}}
\newcommand{\R}{\mathbb{R}}
\newcommand{\C}{\mathbb{C}}
\newcommand{\Z}{\mathbb{Z}}
\newcommand{\Sym}[1]{\mathrm S({#1})}
\newcommand{\SSS}{\mathbb{S}}
\newcommand{\E}{\mathcal{E}}
\newcommand{\K}{K}   
   
\newcommand{\Sp}{\mathrm{Sp}} 
\renewcommand{\L}{L}
\newcommand{\pot}{U}
\newcommand{\pothom}{\pot}
\newcommand{\Lagr}{\Lambda}
\newcommand{\X}{\mathcal{X}}

\newcommand{\Xhat}{\widehat{\X}}
\newcommand{\IAS}{\mathcal{I}}
\newcommand{\Ilim}{\mathcal{I}^*}
\newcommand{\qAS}{\mathcal Q}	
\newcommand{\qlim}{\mathcal{Q}^*}
\newcommand{\Id}{I}
\newcommand{\eps}{\varepsilon}

\newcommand{\BS}{$\mathrm{[BS]}$}
\newcommand{\norm}[1]{\left\| #1 \right\|}  
\newcommand{\Mprod}[2]{ \left\langle {#1},{#2} \right\rangle_M}
\newcommand{\Mnorm}[1]{\abs{#1}_M} 

\newcommand{\abs}[1]{\lvert #1 \rvert}
\newcommand{\trasp}[1]{{#1}^\mathsf{T}}

\newcommand{\iCLM}{\mu}
\newcommand{\iMor}{n_-}
\newcommand{\coiMor}{n_+}
\newcommand{\irel}{I}
\newcommand{\ispec}{\iota_{\textup{spec}}}

\newcommand{\igeo}{\iota_{\textup{geo}}}

\DeclareMathOperator{\diag}{diag}		% diagonal matrix
\DeclareMathOperator{\spfl}{sf}			% spectral flow
\DeclareMathOperator{\sgn}{sgn}		% signature
\renewcommand{\leq}{\leqslant}
\renewcommand{\geq}{\geqslant}
\renewcommand{\hat}{\widehat}

\renewcommand{\=}{\coloneqq}			% definisce :=
\newcommand{\otimesm}{\otimes_M}

\newcommand{\ie}{i.e.~}

\title[An Index theory for asymptotic motions under singular potentials]{An 
Index theory for asymptotic motions \\ under \\ singular potentials}
\author{Vivina L. Barutello, Xijun Hu, Alessandro Portaluri, Susanna 
Terracini}
\date{\today}
\thanks{The second author was partially supported by NSFC (No.11425105 and  No. 11790271).
The first, third and fourth 
author were partially supported by the project ERC Advanced Grant 2013 
n.~339958 ``Complex Patterns for Strongly Interacting Dynamical Systems'' --- 
COMPAT}

\date{\today}
\begin{document}

\begin{abstract}
We develop an index theory for parabolic and collision solutions to the 
classical $n$-body problem and we prove sufficient conditions 
for the finiteness of the spectral index valid in a large class of trajectories ending with 
a total collapse or expanding with vanishing limiting velocities. Both problems 
suffer from a lack of compactness and can be brought in a similar form of a 
Lagrangian System on the half time line by a regularising change of coordinates 
which preserve the Lagrangian structure.  We then introduce a  Maslov-type 
index which is suitable to capture the asymptotic nature of these 
trajectories as half-clinic orbits: by taking into account the underlying  
Hamiltonian structure we define the appropriate  notion of {\em geometric 
index\/} 
for this class of solutions  and we develop  the relative index theory. 
\end{abstract}
 \maketitle
%========================================
%========================================
\section{Introduction}\label{sec:intro}
%========================================
%========================================

Most dynamical systems of interacting particles  in Celestial and 
other areas of Classical  Mechanics,  are governed by singular forces. In 
this paper we are concerned with potential of $n$-body type having the form
\begin{equation}\label{eq:potential-intro}
\pothom(q)\=\sum_{\substack{i,j=1 \\i<j}}^n 
\dfrac{m_im_j}{\norm{q_i-q_j}^\alpha}, 
\end{equation}
where $\alpha \in (0,2)$. The positive real numbers $m_1,\dots,m_n$  can be thought as masses and the  
(real analytic) self-interaction {\em potential function\/} (the opposite to the  potential energy) is defined and smooth in 
$$
\Xhat=\{q=(q_1,\dots,q_n)\in\X\;:\;q_i\neq q_j\; \textrm{if}\; i\neq j\}
$$
where
$$
\X=\{q=(q_1,\dots,q_n)\in\R^{d\times n} \;:\; \sum_{i=1}^{n}m_iq_i=0\}
$$
is the space of admissible configurations with \emph{vanishing 
barycenter}. 
This class of potentials includes the case, for instance,  of the gravitational potential ($\alpha=1$) 
governing the motion of $n$ point masses interacting according to Newton's law of 
gravitation and usually referred to as the (classical) $n$-body problem. Besides 
the  homogeneity of some negative degree, a common feature of all these singular 
potentials is that their singularities are located on some stratified subset (in general 
an arrangement of subspaces or more generally submanifolds) of the full
configuration space.
Both properties play a fundamental role in the study of the dynamics of the 
system and strongly influence the global orbit structure being responsible, 
among others, of the presence of chaotic motions as well as of  motions 
becoming 
unbounded in a finite time. 

Associated with  the potential $\pothom$ we have Newton's Equations of motion 
\begin{equation}\label{eq:NewtonINTRO}
M\,\ddot q = \nabla \pothom(q)
\end{equation}
where $\nabla$ denotes the (Euclidean) gradient and
$M\=\diag{(m_1\Id_d,\dots,m_n\Id_d)}$ is the mass matrix. 
As the centre of mass has an inertial motion, it is not  restrictive to 
prescribe its position at the origin and the choice of $\Xhat$ as  
configuration 
space is justified. For $(q,v)\in T\Xhat$, the tangent bundle of $\Xhat$, the 
Lagrangian function $\L: T\Xhat \to [0,+\infty)\cup\{+\infty\}$ is given by 
\begin{equation}\label{eq:LagrangianaINTRO}
 \L(q,v)=\K(v)+ \pothom(q) = \frac12\Mnorm{v}^2+\pothom(q)\;,
\end{equation}
where $\Mnorm{\cdot}$ denotes the norm induced by the mass matrix $M$.

We deal with the following class of trajectories having a prescribed asymptotic behaviour.
\begin{defn}\label{def:1collisionsolution-intro}
Given $T \in (0,+\infty]$, a {\em $T$-asymptotic solution (or 
simply asymptotic solution or a.s. in shorthand notation) for the dynamical 
system \eqref{eq:NewtonINTRO}\/} is a function $\gamma\in \mathscr 
C^2((0,T),\Xhat)$ which pointwise solves \eqref{eq:NewtonINTRO} on $(0,T)$ 
and such that the following alternative holds: 
\begin{enumerate}
\item[(i)] if $T<+\infty$, then $\gamma \in \mathscr C^0([0,T],\X)$ and it experiences  a \emph{total collision} at the final instant $t=T$, namely  
\[
 \lim_{t \to T^-} \gamma(t)=0,\;  \textit{termed a time-$T$ total collision 
trajectory;}
\]
\item[(ii)] if $T=+\infty$, then $\gamma \in \mathscr C^0([0,+\infty),\X)$ and
\[
 \lim_{t \to +\infty} \K\big(\dot\gamma(t)\big)=0,\;   \textit{termed a completely parabolic trajectory}.
\]
\end{enumerate}
\end{defn}
Such class includes homothetic self-similar motions $\gamma(t)=|\gamma(t)|_M s_0$, where $s_0$ is a central configuration, i.e. a critical point of the potential $\pothom$ constrained on the ellipsoid $\mathcal{E} = \{q \in \Xhat : \Mnorm{q}=1\}$.
The study of these two kinds of motions has occupied a quite extensive research in the field and their strong connection already appeared in Devaney's works \cite{Dev78,Dev81} and it has been clearified recently in \cite {DaLM13} and \cite{BTV13}.

For a.s. some asymptotic estimates, at $t \to T$, are available (cf. \cite{Spe70,Sun13} when $T<+\infty$ and \cite{BTV14} when $T=+\infty$) and one can prove that total collision and completely parabolic trajectories share the same behaviour: the radial component of the motion is, in both cases, asymptotic to the power $\cdot^{\frac{2}{2+\alpha}}$ (infinite or infinitesimal), while its configuration approaches the set of central configurations of the potential at its limiting level (for the details we refer to Lemma \ref{thm:SundmanSperling}).

We shall be concerned with asymptotic solutions having a precise limiting normalized central 
configuration $s_0$, terming these trajectories {\em $s_0$-asymptotic solutions}.
The aim of the present paper is to relate the Morse/Maslov index of the trajectory with its 
asymptotic properties. This will be the starting point for our long term project of building a Morse  Theory, tailored for singular Hamiltonian Systems, which takes into account the 
contribution  of the flow on the collision manifold, defined by McGehee in 
\cite{MG74} by blowing up the singularity. A major problem in using directly 
McGehee regularizing coordinates is that the resulting flow looses its 
Hamiltonian as well as its Lagrangian character. We shall thus use a variant which keeps both the 
Hamiltonian and Lagrangian structure, while both types of asymptotic orbits, total collision and parabolic, will transform into half-clinic trajectories asymptotic to the collision 
manifold.  By taking into account the underlying  Hamiltonian structure we will introduce appropriate 
notions of {\em spectral index\/} and {\em geometric index\/} for this class of half-clinic orbits and we will relate them through an Index Theorem. 
These two topological invariants  respectively encode, in the index formula, 
the {\em spectral terms} which are  computationally  inaccessible, and the  
{\em geometric terms\/} 
containing analytic  information of a fundamental nature  which are quite explicit and involve the spectrum of a finite dimensional 
operator.

To the authors knowledge, there are very few results in literature that investigate the contribution of collisions to the Morse index of a variational solution when the potential is homogeneous and  weakly singular. In a very recent paper G. Yu (inspired by some paper by Tanaka, we cite for all \cite{TanAIHP93}) proved that the Morse index of a solution gives an upper bound on the number of binary collision of the $n$-body problem (cf. \cite{Yupreprint}). The only partial results that takes into account collisions involving more then two bodies are contained in the pioneering papers \cite{BS08, HO16}, in which the authors, using the precise asymptotic estimates near collisions, prove that  the {\em collision  
Morse index\/} is infinite whenever the smallest eigenvalue of $ 
D^2\pot|_\mathcal E(s_0) $ is less than 
a threshold depending on the homogeneity $\alpha$ and $s_0$ itself. 

In the present paper the following spectral condition, naturally  associated with the $s_0$-asymptotic solution, plays a central role
\begin{equation}\label{eq:BS}
\text{the smallest eigenvalue of } D^2\pot|_\mathcal E(s_0) \text{ is }> 
 -\dfrac{(2-\alpha)^2}8{\pot(s_0)}.
\end{equation}
We will refer to this condition as the \textit{\BS-condition}. An analogue spectral condition has been recently used in \cite{Hua11} to prove non-minimality of a class of collision motions in the planar Newtonian three body problem, expressed using Moser coordinates.  It is worthwhile noticing that the \BS-condition has an important dynamical interpretation, marking the threshold for  hyperbolicity of $s_0$ as 
a  rest point on the flow restricted to the collision manifold.

In one of the main result of the present paper, Theorem \ref{prop:key1-una freccia}, we prove 
that when the central configuration $s_0$ satisfies the \BS-condition then the Morse Index of any $s_0$-a.s. is finite.
A direct consequence of this theorem is 
that the condition given by authors in \cite{BS08} is sharp for the infiniteness of the (collision) 
Morse index.
Actually the infiniteness of the Morse index is paradigmatic of a more general situation of boundary value problems for systems of ordinary differential 
equations on the half-line  (cf. \cite{RS05a, RS05b}). In fact, we shall prove 
in Theorem \ref{thm:key1fredholm} that the \BS-condition gives the threshold for the 
linearized operator belonging to the Fredholm class. 
We observe that in the particular case in which the smallest eigenvalue 
of $ D^2\pot|_\mathcal  E(s_0)$ is equal to 
$-\frac{(2-\alpha)^2}{8}\pot(s_0)$, then the Morse index could be finite even if the 
associated index form is not Fredholm.

The class of Fredholm operators has been the 
natural environment, since their debut, for index theories, whose principal protagonist is 
represented by the {\em spectral flow\/} for 
paths of selfadjoint Fredholm operators. This celebrated topological invariant was introduced by Atiyah, Patodi and Singer in the seminal paper \cite{APS76} and since then it has reappeared in connection with several other phenomena like odd Chern characters, gauge anomalies, Floer homology, the 
distribution of the eigenvalues of the Dirac operators. In  the 
finite-dimensional context the spectral flow probably dates even back to Morse 
and his index theorem while, in the one-dimensional case, it traces back to Sturm.  There are 
several different definitions of the spectral flow that appeared in the literature during the last decades. 

The last section of the paper is devoted to prove a Morse type Index Theorem for $s_0$-asymptotic solutions. This results states that, when \BS-condition is satisfied, then the \emph{spectral (or Morse) index} and the \emph{geometric (or Maslov) index} of an $s_0$-a-s. coincide. The equality between these two topological invariants allows us to mirror the problem of computing the Morse index of a $s_0$-a.s. (integer associated to an unbounded Fredholm operator
in an {\em infinite dimensional\/} separable Hilbert space) into an intersection index between a curve of Lagrangian subspaces and a {\em finite dimensional\/} transversally oriented variety. 
The key of this result relies on the fact that the spectral index can be 
related to the spectral flow of a path of Fredholm quadric forms. The interest for this spectral flow formula is twofold. On the one hand, it could be useful for the computation of the Morse index since the computation is confined on a finite dimensional objects; on the other hand, from a theoretical point of view, it could be interesting to prove a Sturm oscillation theorem for singular systems and to relate the geometrical index with the Weyl classification of the singular boundary condition (cf. \cite{Zet05} and references therein). To the authors knowledge, very few spectral flow formulas have been (recently) proved in the case of Lagrangian and Hamiltonian systems defined on an unbounded interval and in particular  on half-line. (We refer the interested reader to \cite{CH07},\cite{HP17} and references therein). For this reason, we think also that, the interests of the spectral flow formula proved in Theorem \ref{thm:indextheorem} goes even beyond the framework of the present paper. 

We finally remark that in the case of collisionless periodic solutions, in the  
last few years the Morse Index of some special equivariant orbits has been 
computed via Index Theory: in \cite{HS10} and \cite{HLS14} the authors studied 
the Morse Index of the family of elliptic Lagrangian solutions of the planar 
classical three body problem, while in \cite{BJP16, BJP14} the authors restrict to the 
circular Lagrangian motions, taking into account different homogeneity 
parameter  of the potential. 
The first step in the rigorous analysis of the relation between the symmetry 
of  the orbit and its Morse Index (still via Index Theory) have been performed by 
Hu and  Sun in \cite{HS09} in which the authors computed the Morse Index and 
studied the stability of the figure-eight orbit for the planar 3-body problem 
by means of symplectic techniques.

We conclude the introduction,  by observing that we have only scratched the surface of the 
implication of the Maslov index for halfclinic trajectories in Celestial Mechanics or 
more generally in singular Hamiltonian systems. It is reasonable to 
conjecture that the index theory  constructed in this  paper for $s_0$-a.s. could be carried 
over for more general class of asymptotic motions arising in a natural way in    
Celestial Mechanics; more precisely, motions in  which the limit of the normalized central configuration does  not hold, but the linearized system along the motion experiences an exponential  dichotomy as well as for orbits experiencing only 
partial collisions. Both these problems are promising research directions that should be 
investigated in order to obtain a deeper understanding of the intricate dynamics governed by singular potentials.
\vskip1truecm

For the sake of the reader, in the following list we collected some mathematical symbols and constants that will appeared in the paper. 
\begin{itemize}
	\item $\alpha \in (0,2)$ is the homogeneity of the potential $\pot$
	\item $M$ is the mass matrix
	\item $\mathcal E$ is the \textrm{ inertia ellipsoid } with respect to the mass metric
	\item  $s_0$ is  a \textrm{ normalized central configuration}
    \item $\beta:=\dfrac{2(2+\alpha)}{2-\alpha} >2$
    \item $c_\alpha:=\left(\dfrac{4}{2-\alpha}\right)^2-1$
    \item $\overline \delta_\alpha:=
    \begin{cases} 
    	\small{-\dfrac{(2-\alpha)}{4} \sqrt{2\pot(s_0)}}, & \text{if } T<+\infty \\
    	\small{ \dfrac{(2-\alpha)}{4} \sqrt{2\pot(s_0)}}, & \text{if }  T=+\infty
    \end{cases}$
    \item $c_\alpha \overline \delta_\alpha^2 = 2 \pot(s_0)-
    \dfrac{(2-\alpha)^2}{8} \pot(s_0)$ 
    \item $\widetilde \delta_\alpha:=\overline \delta_\alpha (\beta-2)= 
    -\alpha \sqrt{2\pot(s_0)}$
\end{itemize}

\vspace{.5cm}

The paper is organised as follows: 
\tableofcontents

%========================================
%========================================
\section{Variational setting and regularised action functional} 
\label{sec:variational_setting}
%========================================
%========================================

In this section  we settle the variational framework of the problem. In 
Subsection \ref{subsec:time1cs} we introduce  the main definitions and the basic 
notation which we need in the sequel as well as we define  a suitable class of 
{\em asymptotic motions\/} both collision and parabolic. Subsection \ref{subsec:var} is devoted to fix the variational setting. In Subsections 
\ref{subsec:McGehee} and \ref{subsec:firstsecondvariation}, by means of the {\em  Lagrangian version of McGehee's 
transformation\/}, we compute the second variation of the corresponding 
Lagrangian action functional in these new coordinates. 

%========================================
\subsection{Asymptotic motions: total collisions and parabolic 
trajectories}\label{subsec:time1cs}
%========================================

For any integer $n \geq 2$, let $m_1, \dots, m_n$ be $n$ positive real numbers 
(that can be thought as the masses of  $n$ point particles) and let $M$ be the diagonal block matrix defined  as $M\=[M_{ij}]_{i,j=1}^n$ with 
$M_{ij}\=m_j\delta_{ij}\Id_d$, where $\Id_d$ is the $d$-dimensional identity matrix and $d \geq 2$. Being $\langle \cdot, \cdot \rangle$ 
the Euclidean scalar product in $\R^{nd}$, we indicate with 
\begin{equation}\label{eq:massnorm}
\Mprod{\cdot}{\cdot} = \langle M\cdot ,\cdot \rangle 
\qquad \text{and}  \qquad
\Mnorm{\cdot} =  \langle M\cdot,\cdot \rangle^{1/2}
\end{equation}
respectively the Riemannian metric and the norm induced by the mass matrix. For 
brevity we shall refer to them respectively as the {\em mass scalar product\/} 
and the {\em mass norm\/}.

Let $\X$ denotes the {\em configuration space\/} of the $n$ point particles 
with 
masses $m_1, \dots, m_n$ and centre of mass in $0$
\begin{equation}\label{eq:configuration-space}
\X \= \Set{(q_1, \dots, q_n) \in \R^{nd}| \sum_{i=1}^n m_i\, q_i=0}.
\end{equation}
Thus $\X$ is a $N$-dimensional (real) vector space, where $N\=d(n-1)$. For each 
pair of indices $i,j \in \{1,\ldots,n\}$ let $\Delta_{i,j}=\Set{x \in \X| 
q_i=q_j}$ be the {\em collision set of the $i$-th and $j$-th particles\/}
and let 
\begin{equation}\label{eq:singular-set}
\Delta\= \bigcup_{\substack{i,j=1\\i \neq j}}^n \Delta_{i,j}
\end{equation}
be the {\em collision set in  $\X$\/}. It turns out that $\Delta$ is a {\em cone\/} whose vertex is the point $0 \in \Delta$; it corresponds to the \emph{total collision} or to the {\em total collapse\/} of the system (being  the centre of mass  fixed at $0$). The space of \emph{collision free configurations} is denoted by
\[
\Xhat := \X \setminus \Delta.
\]

For any real number $\alpha \in (0,2)$ and any pair $i,j \in \{1,\ldots,n\}$, 
we 
consider the smooth  function $\pot_{ij}:\R^d\setminus (0) \to [0,+\infty)$ given by 
\[
 \pot_{ij}(z)\=\dfrac{m_im_j}{\norm{z}^\alpha}.
\]
We define the (real analytic) self-interaction {\em potential function\/} on 
$\Xhat$ (the opposite of the potential energy),  $\pothom: \Xhat\to [0,+\infty)$ as 
follows \begin{equation*}\label{eq:potential}
\pothom(q)\=\sum_{\substack{i,j=1 \\i<j}}^n \pot_{ij}(q_i-q_j).
\end{equation*}
Newton's Equations  of motion associated to $\pothom$ are 
\begin{equation}\label{eq:Newton}
M\,\ddot q = \nabla \pothom(q)
\end{equation}
where $\nabla$ denotes the gradient with respect to the Euclidean  metric. 
It turns out that, since the centre of mass has an inertial motion, it is not 
restrictive to fix it at the origin and the choice of $\X$ as  
configuration space is justified. Denoting by $T\X$ the tangent bundle of $\X$, 
whose elements are denoted by $(q,v)$ with $q \in \X$ and $v$  a tangent vector 
at $q$,  the Lagrangian function $\L: T\X \to [0,+\infty)\cup\{+\infty\}$ of the 
system is
\begin{equation}\label{eq:Lagrangiana}
 \L(q,v)=\K(v)+ \pothom(q) 
\end{equation}
where the first term is the {\em kinetic energy\/} $\K(v)\= \frac12\Mnorm{v}^2$ 
of the system.

We will deal with $T$-asymptotic solutions as introduced in Definition \ref{def:1collisionsolution-intro}. 
Concerning the definition of completely parabolic motions let us recall that, by using the concavity of the second derivative of the moment of inertia, one can prove that completely parabolic motions have necessarely zero energy (cf. \cite[Definition 1 and Lemma 2]{Che98}). For this reason, from now on we will name this kind of asymptotic motions simply \emph{parabolic} trajectories.

The behaviour of asymptotic motions as $t$ approaches to $T$ is well-known (cf. 
\cite{Sun13,Spe70, BFT08, BTV14} and references therein) and in order to 
describe it, we perform a polar change of coordinates in the configuration 
space 
with respect to the mass scalar product. Let $r$ and $s$ be respectively the {\em  radial\/} and {\em angular\/}  variables associated to a configuration $q \in \X$: 
\begin{equation}
\label{eq:rs}
r \= \Mnorm{q} \in [0,+\infty), \qquad s \= \frac{q}{r} \in  \E 
\end{equation}
where 
$$
\E  = \{q \in \Xhat : \Mnorm{q}=1\}
$$
is the {\em inertia ellipsoid\/}, namely the unitary sphere in the 
mass scalar product. 
\begin{rem}\label{rem:no_coll}
In this remark we aim to motivate the fact that the assumption in Definition 
\ref{def:1collisionsolution-intro}
\begin{equation} \label{eq:cond_def}
\gamma(t) \in \Xhat, \;\forall t \in (0,T),
\end{equation}
is not too restrictive.
Let us start considering a time-$T$ total collision trajectories; since the centre of mass of the system has been fixed at the 
origin, assumption given in Formula \eqref{eq:cond_def} implies the following 
condition on the 
radial component,
$r(t) = |\gamma(t)|_M$,
\[
\lim_{t\to T^-} r(t)=0 \;\text{and}\; r(t) \neq 0, \quad  \forall\, t \in [0,T).
\]
This fact is indeed a consequence of the Lagrange-Jacobi inequality (which actually shows 
the convexity of $|r(t)|_M^2$ in a neighbourhood of the total collision). 
Furthermore, the authors in \cite{BFT08} prove that a total collision is indeed isolated \emph{among  collisions} (not only total, but also partial).  For these reasons the class of total collision motions we choose is a natural set of total collision motions. 

When $\gamma$ is a (completely) parabolic motion, since the kinetic energy goes 
to 0 and the total energy is conserved then, as $T \to +\infty$, each
mutual distance between pairs of bodies is bounded away from 0. Being $\Xhat$ 
an 
open set, also in this case assumption \eqref{eq:cond_def} is not restrictive.
\end{rem}
The asymptotic behaviour both of total collision and parabolic motions is 
described by the forthcoming Lemma \ref{thm:SundmanSperling}. For its proof  we refer the interested reader respectively to 
\cite[Theorem 4.18]{BFT08} and \cite[Theorem 
7.7]{BTV14}. 
We recall the following definition.
\begin{defn}\label{def:cc}
A configuration $s_0 \in \Xhat$ is a \emph{central configuration} for the 
potential $\pothom$ if it is a critical point for $\pothom|_{\mathcal{E}}$. Any central configuration 
$s_0$ for $\pothom$, homogeneous of degree $-\alpha$, satisfies the central 
configuration equation
\begin{equation}\label{eq:cc_eq}
\nabla \pothom(s_0) = -\alpha \pothom(s_0)Ms_0.
\end{equation}
\end{defn}
\begin{lem}{\bf[Sundman-Sperling, \cite{Spe70,Sun13} and, for the case $T=+\infty$, \cite{BTV14}]\/}\\ \label{thm:SundmanSperling}
Let us assume that $\gamma$ is an asymptotic solution for the dynamical system 
\eqref{eq:Newton} and define 
\[
K\= \dfrac{\alpha+2}{\alpha}\sqrt{2\,b}
\quad \text{and} \quad
\beta(t)\=
\begin{cases}
T-t, & \textrm{ if }\  T<+\infty\\
t,   & \textrm{ if }\,  T=+\infty.
\end{cases}
\]
Then there exists $b>0$ such that the following estimates 
hold:
\begin{enumerate}
\item[(i)] $\displaystyle \lim_{t \to T^-} r(t)[K\beta(t)]^{-\frac{2}{2+\alpha}} = 1 $ \quad and  
\[
\lim_{t \to T^-} \dot r(t)[K\beta(t)]^{\alpha/(2+\alpha)} = 
\begin{cases}
- \sqrt{2\,b}, & \textrm{ if }\  T<+\infty\\
\sqrt{2\,b},   & \textrm{ if }\,  T=+\infty.
\end{cases}
\] 
\item[(ii)]  
\(
\displaystyle
\lim_{t \to T^-} \pothom(s(t)) = b.
\)
\item[(iii)] 
\[
\lim_{t \to T^-} \Mnorm{\dot s(t)}\, \beta(t)= 0\quad \textrm{ and  } \quad 
\lim_{t \to T^-} \mathrm{dist}(\mathcal C^b, s(\tau))=0 
\]
where $\mathcal C^b$ is the set of central configurations for $\pothom$ at 
level 
$b$.
\end{enumerate}
\end{lem}
The previous lemma motivates the choice to	gather in only one definition 
total collision motions and parabolic ones: their asymptotic behaviour is  really 
similar when $t$ approaches to $T$. In particular let us pause on assertion (iii) 
of Lemma \ref{thm:SundmanSperling}: it states that the angular part of these 
trajectories tends to a set $\mathcal{C}^b$, but it does not necessarily admit a 
limit. The absence (or presence) of spin in the angular part of an asymptotic 
motion is indeed still an open problem in Celestial Mechanics, known as the 
{\em infinite spin problem} (cf. \cite{SH81,Saa84}). 
Our studies will focus on a.s. admitting a limiting central configuration: 
we are interested in collision/parabolic 
trajectories whose angular part tends exactly to an assigned central configuration $s_0$.
\begin{defn}\label{def:s_0-as}
Given a central configuration $s_0$ for the potential $\pot$ and an a.s. 
$\gamma$, we say that  $\gamma$ is an \emph{$s_0$-asymptotic solution} 
(\emph{$s_0$-a.s.} for short) if 
\[
\lim_{t \to T^-}\frac{\gamma(t)}{|\gamma(t)|_M} = s_0. 
\]
\end{defn}
\begin{rem}[\emph{on $s_0$-homothetic motions}]\label{rem:hom}
	Given a central configurations for $\pothom$, $s_0$, the set of $s_0$-a.s. is not empty. It indeed contains the {\em $s_0$-homothetic motions}: these self-similar trajectories has the form $\gamma_0(t)=r_0(t)s_0$ where $r_0(t)$ solves the 1-dimensional Kepler problem. If the constant energy, $h$, is negative these solutions are bounded, ending in a collision as $t \to T<+\infty$ (actually these motions can be extended to $(-T,T)$). When $h \geq 0$ they are unbounded motions starting or ending with a total collision. {When $h=0$ we find an \emph{$s_0$-homothetic parabolic motion}, $\gamma^*(t)$: the radial part of these trajectories can be computed explicitely and there holds
	\[
	\gamma^*(t) = [K\beta(t)]^{\frac{2}{2+\alpha}}s_0,
	\]
	where the constant $K$ and the function $\beta$ have been introduced in Lemma \ref{thm:SundmanSperling}.}
	When $\beta(t) =t$ an $s_0$-homothetic parabolic motion $\gamma^*$ borns in a collision and becomes unbounded as $t\to T=+\infty$; when $\beta(t) = (T-t)$ the trajectory ends in a total collision and can be extended to the maximal definition interval $(-\infty,T)$, becoming unbonded as $t \to -\infty$.
	Hence Lemma \ref{thm:SundmanSperling} states that $s_0$-asymptotic solutions are indeed, as $t \to T$, perturbations of $s_0$-homothetic parabolic motions.
\end{rem}	
\begin{rem}
	Whenever $s_0$ is a minimal configuration, existence of $s_0$-asymptotic arcs of minimal  parabolic trajectories departing from arbitrary configurations has been proved in \cite{MV09}. The set of $s_0$-a.s. also contains this kind of trajectories.
	\end{rem}
%
%========================================
\subsection{The variational framework}\label{subsec:var}
%========================================
%
{\bf Collision trajectories.} Let us now take $T \in (0,+\infty)$; it is indeed 
well known that \emph{classical solutions} to Newton's  equations 
\eqref{eq:Newton} (i.e. trajectories in $\mathscr{C}^2((0,T),\Xhat) \cap 
\mathscr{C}^0([0,T],\X)$) can be found as critical points of the {\em 
Lagrangian action functional} associated to the Lagrangian $\L$ (introduced in Eq. 
\eqref{eq:Lagrangiana}) 
\begin{equation*}
\SSS(\gamma) \=  
\int_0^T  \L\big(\gamma(t), \dot \gamma(t)\big)\, dt,
\end{equation*}
in the smooth Hilbert manifold of all $W^{1,2}$-paths in $\X$ with some 
boundary 
conditions (fixed end points or periodic).
Let us then introduce
\[
\Omega := W^{1,2}([0,T],\X),
\]
with scalar product pointwise induced by the mass  scalar product.
Since we are interested in solutions satisfying some boundary conditions, for 
any $p,q \in \X$, we consider the closed linear submanifold of $\Omega$
\[
\Omega_{p,q}\=\Set{\gamma \in \Omega| \gamma(0)=p, \ \gamma(T)=q}.
\]
Given $\gamma \in \Omega$, we will  denote by $W^{1,2}(\gamma)$ the Hilbert 
space of all $W^{1,2}$-vector fields along $\gamma$:
\[
 W^{1,2}(\gamma)\= \Set{\bar\xi \in W^{1,2}([0,T],T\X)| \bar 
\xi(t)=(\gamma(t),\xi(t)), t \in [0,T]}.
\]
It is well-known that the tangent space at $\gamma$ to 
$\Omega$ can be identified with $W^{1,2}(\gamma)$.

When  $\gamma \in  \Omega_{p,q}$, then the tangent space at $\gamma$ to  $\Omega_{p,q}$ is the subspace $W^{1,2}_0(\gamma)$ of 
$W^{1,2}(\gamma)$ defined by 
\[
 W^{1,2}_0(\gamma)\= \Set{\bar\xi \in W^{1,2}(\gamma)|\xi(0) =0 = \xi(T)}.
\]
As before, there exists an identification between $T_{\gamma} \Omega_{p,q}$ with $ W^{1,2}_0(\gamma)$.
This space admits as a dense subset the space of 
smooth functions $\mathscr C^\infty_c(0,T;\X)$. 

In order to characterize asymptotic solutions as critical points of $\SSS$, we 
need to work in the general setting of non-smooth critical point theory, for  this 
reason we recall the following Definition.
\begin{defn}\label{def:differentiability}
Let $X$ be a Hilbert space, $\mathcal J$ a continuous functional on $X$ and let 
$Y$ be a dense subspace of $X$.
If 
\begin{itemize}
\item the directional derivative of $\mathcal J$ exists for all $x \in X$ in 
all 
direction $y \in Y$ (i.e. $D\mathcal J(x)[y]$ exists for all $x \in X$ and $y 
\in Y$), we say that $\mathcal J$ is {\em G\^ateaux $Y$-differentiable\/} and we 
will denote the G\^ateaux $Y$-differential with $D_G\vert_Y$. If 
$\mathcal J$ is G\^ateaux $Y$-differentiable, a point $x \in X$ is said to be a 
{\em  G\^ateaux $Y$-critical point\/} if 
\[
D_G\vert_Y\mathcal J(x)[y]=0 \quad \textrm{ for all } y \in Y;
\]
\item there exists a bounded linear functional $A: Y \to \R$ such that 
\[
 \mathcal J(x+h)-\mathcal J(x)= A\,h + o(h)
\]
for all $h \in Y$, we say  that $\mathcal J$ is {\em Fréchet 
$Y$-differentiable\/} and we denote the $Y$-Fréchet differential at $x$, $A$, 
by 
$D_Y\mathcal J(x)$. If $\mathcal J$ is Fréchet $Y$-differentiable, a point $x 
\in X$ is said to be a {\em   $Y$-critical point\/} if 
\[
D_Y\mathcal J(x)[y]=0 \quad \textrm{ for all } y \in Y. 
\]
\end{itemize}
\end{defn}
In the next result we characterise weak solutions of Equation \eqref{eq:Newton} 
(with fixed ends) as G\^ateaux $\mathscr C^\infty$-critical  points (where $\mathscr C^\infty_c$ = $\mathscr C^\infty_c(0,T;\R^{N})$) of the Lagrangian action functional. 
\begin{lem}\label{le:weak_sol}
Let $p \in \Xhat$. Then: 
\begin{enumerate}
\item[(i)] if $\gamma$ is a time-$T$ total collision solution for 
\eqref{eq:Newton} such that $\gamma(0)=p$,
then $\gamma \in\Omega_{p,0}$, $\SSS(\gamma)<+\infty$, and $\gamma$ is a 
$W^{1,2}$-solution (or \emph{weak solution}) for \eqref{eq:Newton}, that is
\begin{equation}\label{eq:weak}
\int_0^T \Mprod{\dot \gamma}{\dot \xi} = - \int_0^T\Mprod{\nabla 
\pothom(\gamma)}{\xi},
\qquad \forall \xi \in \mathscr {C}^\infty_0(0,T;\X).
\end{equation}
\item[(ii)] If $\gamma \in \Omega_{p,0}$, $\gamma|_{(0,T)} \subset \Xhat$, 
satisfies \eqref{eq:weak}, then $\SSS:\Omega_{p,0}\to [0,+\infty)\cup\{+\infty\}$ is 
G\^ateaux $\mathscr C^\infty_c$-differentiable at $\gamma$ and $\gamma$ is a 
G\^ateaux $\mathscr C^\infty_c$-critical point of $\SSS$ and a classical solution 
of 
\eqref{eq:Newton} on $(0,T)$.
\end{enumerate}
\end{lem}
\begin{proof}
\emph{(i)} As a direct consequence of item (i) in Lemma 
\ref{thm:SundmanSperling}, the following 
pointwise asymptotic behaviour in the neighbourhood of the final instant $t=T$ 
holds:
\[
 \Mnorm{\dot \gamma(t)}^2 \sim [K(T-t)]^{-\frac{2\alpha}{2+\alpha}}
 \textrm{ and }
 \pothom(\gamma(t))\sim [K(T-t)]^{-\frac{2\alpha}{2+\alpha}}\,\pothom(s(t)).
\]
By using item (ii) in Lemma \ref{thm:SundmanSperling}, we obtain $\pothom(\gamma(t)) \sim [K(T-t)]^{-\frac{2\alpha}{2+\alpha}}b$.

Combining the first estimate with the continuity of $\gamma$ on $[0,T]$, we 
deduce that $\gamma \in \Omega_{p,0}$. Furthermore, by taking into account that the 
image of the function $ g:(0,2) \to \R$ defined by 
$g(\alpha)\=-2\alpha/(2+\alpha)$ is contained in the interval $(-1,0)$, it 
follows that, the integral $\SSS$ along any time-$T$ total collision solutions 
converges. This proves that  $\SSS(\gamma)$ is finite. 
Equation \eqref{eq:weak} is the weak formulation of Equation \eqref{eq:Newton} for the solution $\gamma$.\\
\emph{(ii)} Let $\xi \in \mathscr C^\infty_c(0,T;\X)$, we compute
\begin{equation}\label{eq:var-prima}
\dfrac{1}{h}\big(\SSS(\gamma+h\xi)- \SSS(\gamma)\big)
= \int_{\text{supp}\xi}\left\{ \Mprod{\dot \gamma}{\dot \xi} +\frac12 h |\dot 
\xi|_M^2 +\frac1h \left[U(\gamma +h\xi)-U(\gamma)\right]\right\}.
\end{equation}
Since $\gamma|_{(0,T)} \subset \Xhat$, $\xi \in \mathscr C^\infty_c(0,T; 
\X)$, whenever $h \in \R$ has a 
small absolute value, then $\gamma(t) +h\xi(t)$ still belong to the open set 
$\Xhat$ and, by dominated convergence 
\[
d\,\SSS(\gamma)[\xi] = \lim_{h \to 0} \dfrac{1}{h}\big(\SSS(\gamma+h\xi)- 
\SSS(\gamma)\big) = \int_{\text{supp}\xi}\left\{ \Mprod{\dot \gamma}{\dot 
\xi} 
+ \Mprod{\nabla U(\gamma )}{\xi}\right\}.
\]
Hence $d\, \SSS(\gamma)$ is a bounded 
linear functional when regarded on the densely immersed subspace $\mathscr 
C^\infty_c(0,T;\X)$ of $W_0^{1,2}(\gamma)
$, $\SSS$ is G\^ateaux-$\mathscr C^\infty_c$-differentiable and $d\, 
\SSS(\gamma)$ is its G\^ateaux 
differential at $\gamma$. 
In particular $d\,\SSS(\gamma)[\xi]=0$ for any $\xi \in \mathscr C^\infty_c$ 
means that $\gamma$ is a  G\^ateaux $\mathscr C^\infty_c$-critical point in the 
sense specified in Definition \ref{def:differentiability}.
We conclude by standard elliptic regularity arguments.
\end{proof}

{\bf Parabolic motions.} First of all we remark that, even if parabolic motions 
do not interact with the singular set (cf. Remark \ref{rem:no_coll}), they are 
unbounded motions, they miss the integrability properties at infinity, hence 
they don't lie in the set $W^{1,2}([0,+\infty),\Xhat)$ but just in $W_{loc}^{1,2}([0,+\infty),\Xhat)$. Although the functional $\SSS$ is infinite along a parabolic motion $\gamma$, it can be G\^ateaux $\mathscr C^\infty_c$-differentiable at $\gamma$: indeed the difference in Equation \eqref{eq:var-prima} turns out to be finite for every $\xi \in \mathscr C^\infty_c$. {In a more abstract way, it is possible to define the functional $\SSS$ on the Hilbert manifold $\gamma + W^{1,2}([0,+\infty),\Xhat)$.} Hence we can argue as in Lemma \ref{le:weak_sol} to deduce the next result.
\begin{lem}\label{thm:azione-finita}
Let $\gamma \in W_{loc}^{1,2}([0,+\infty),\Xhat)$ be such that for any $T>0$ Equation \eqref{eq:weak} is fulfilled. Furthermore let us assume that 
\[
 \lim_{t \to +\infty} \K\big(\dot\gamma(t)\big)=0;
\]
then the Lagrangian action functional
\[
\SSS(\gamma) \=  
\int_0^{+\infty}  \L\big(\gamma(t), \dot \gamma(t)\big)\, dt
\]
is infinite and it is G\^ateaux $\mathscr C^\infty_c$-differential at $\gamma$. 
Furthermore  $\gamma$ is a $\mathscr C^\infty_c$-critical point of $\SSS$ and 
it is indeed a classical solution of the Equation  \eqref{eq:Newton} on 
$(0,+\infty)$.
\end{lem}

%========================================
\subsection{A variational version of McGehee 
regularisation}\label{subsec:McGehee}
%========================================

The Lagrangian function $\L$ defined in Equation \eqref{eq:Lagrangiana} can be 
written 
in term of the radial and angular coordinates $
(r,s) \in \mathcal{N}\=[0,+\infty) \times \mathcal{E}$ introduced in Formula 
\eqref{eq:rs}.  
Indeed $\L$ transforms into \[
\bar\L(r,\dot r,s,\dot s) \= \dfrac12 \dot r^2 + \dfrac12 r^2\Mnorm{\dot s}^2 + 
\frac{1}{r^{\alpha}}\pothom(s).
\]
which is defined on $T\hat{\mathcal N}$, where 
\[
\hat{\mathcal N} := {\mathcal N} \setminus\bar \Delta,
\]
and $\bar \Delta$ corresponds to the singular set $\Delta$ in the new 
variables, 
that is
\[
\bar \Delta := \{(0,s) : s \in \mathcal{E}\}\cup\{ (r,s) : r \in (0,+\infty),\, 
s \in \Delta \cap \mathcal{E}\}.
\]

Let us now consider an $s_0$-a.s. $\gamma$ for the dynamical system 
\eqref{eq:Newton} and let $r=r(t)$ and $s=s(t)$ be its radial and angular 
components; inspired by the behaviour of $r(t)$ when $t \to T$,  (cf. Lemma 
\ref{thm:SundmanSperling}, (i)), we introduce the new time-variable 
\begin{equation}\label{eq:tau}
\tau =\tau(t) :=\int_0^t r^{-\frac{2+\alpha}{2}}.
\end{equation}
The function $\tau(t)$ is strictly monotone increasing and $\tau(0)=0$; 
furthermore since as $t \to T^-$ 
\begin{itemize}
\item $r^{-\frac{2+\alpha}{2}}(t) \sim \frac{1}{K(T-t)}$, when $T<+\infty$ (\ie 
$\gamma$ is a time-$T$ total collision motion),
\item $r^{-\frac{2+\alpha}{2}}(t) \sim \frac{1}{Kt}$ when $T=+\infty$ (\ie 
$\gamma$ is parabolic motion),
\end{itemize}
then in both cases the variable $\tau$ varies on the half line, indeed
\[
\int_{0}^{T} r^{-\frac{2+\alpha}{2}} = +\infty.
\]
We now define the variable $\rho \in [0,+\infty)$ as 
\begin{equation}\label{eq:rho}
\rho(\tau) \=r^{\frac{2-\alpha}{4}}(t(\tau)),
\end{equation}
and, denoting by $'$ the derivative with respect to $\tau$ we have
\[
\rho'(\tau) = \dfrac{2-\alpha}{4} r^{-\frac{2+\alpha}{4}}(t(\tau))r'(t(\tau)).
\]
Since 
\begin{equation}\label{eq:constraint_tempo}
 dt = r^{\frac{2+\alpha}{2}} \, d\tau,
\end{equation}
we have 
\[
\rho'(\tau) = \dfrac{2-\alpha}{4} r^{-\frac{2+\alpha}{4}}(t(\tau))\dot 
r(t(\tau))r^{\frac{2+\alpha}{2}}(t(\tau)).
\]
Using Lemma \ref{thm:SundmanSperling} (i) and (iii), we deduce that in the new time variable $\tau$, both classes of $s_0$-a.s., collision (resp. parabolic), share the same asymptotic behavior:
the variable $\rho$ has an exponential decay (resp. growth) while the speed of the angular part tends to 0. More precisely 
\begin{equation}\label{eq:stima_rho'/rho}
\lim_{\tau \to +\infty} \dfrac{\rho'(\tau)}{\rho(\tau)} = \overline \delta_\alpha,
\end{equation}
where the constant ${\overline\delta}_\alpha$ is
\[
\bar \delta_\alpha \= 
\begin{cases}
- \dfrac{2-\alpha}{4}\sqrt{2\pot(s_0)}, & \textrm{ if }\  T<+\infty\\
  \dfrac{2-\alpha}{4}\sqrt{2\pot(s_0)}, & \textrm{ if }\,  T=+\infty
\end{cases},
\]
while
\begin{equation}\label{eq:stima_s'}
\lim_{\tau \to +\infty} |s'(\tau)| = 0.
\end{equation}
\begin{rem}[\emph{on $s_0$-homothetic-parabolic motions}]\label{rem:hom_rho}
When we consider an $s_0$-homothetic parabolic motion $\gamma^*$ (cf. Remark \ref{rem:hom}) and its 
corresponding radial variable $\rho_0(\tau)$, then the following quantity remains constant along the trajectory
\[
\dfrac{\rho_0'(\tau)}{\rho_0(\tau)} = \bar \delta_\alpha,
\]
so that $\rho_0(\tau)$ has an exponential behavior: $\rho_0$ decays if $T$ is finite, increases when $T=+\infty$. Let us recall that in both cases the new time variable $\tau$ varies on $[0,+\infty)$.
\end{rem}
In these new coordinates $(\rho,s) \in \mathcal{N}$, the Lagrangian $\L$ 
transforms into $\hat\L: T\hat{\mathcal N}\to \R$ and reads as 
\begin{equation}\label{eq:densityrhos}
\hat\L(\rho,s,\rho',s')\= \dfrac{1}{2}\left(\dfrac{4}{2-\alpha}\right)^2 
\rho'^2 
+ \rho^2 \left(\dfrac{1}{2}\Mnorm{s'}^2 + \pothom(s)\right).
\end{equation}
We note that, when we have fixed a finite $T>0$ (in the case of total collision 
motions), the time scaling given in Eq.~\eqref{eq:constraint_tempo} imposes an 
infinite dimensional constraint given by 
\begin{equation}
\label{eq:infiniteconstraint}
\|\rho^\beta\|_{L^1(0,+\infty)} = T, 
\qquad \textrm{ for } \beta\= 
\dfrac{2(2+\alpha)}{2-\alpha}>2.
\end{equation}

Dealing with the Lagrangian $\hat L$ one have to 
take into account also the pointwise holonomic  constraint $|s|_M^2-1=0$. In order to prevent the presence of this second constraint, it is 
convenient to define
\begin{equation}\label{eq:y}
y(\tau) = \rho(\tau) s(\tau),
\qquad \tau \in [0,+\infty),
\end{equation}
that actually satisfies the relations 
\[
\Mnorm{y'}^2 = (\rho')^2 + \rho^2 \Mnorm{s'}^2 \text{ and } \Mnorm{y}'=\rho'.
\]
The Lagrangian function $\L : T\Xhat \to \R$ in the variable 
$y$ becomes
\[
\L(y, y')= \dfrac{c_\alpha}{2} (\Mnorm{y}')^2 + \dfrac12 \left[ \Mnorm{y'}^2 + 
2\pothom(s)\Mnorm{y}^2\right],
\]
where
\[
c_{\alpha} \= \left(\dfrac{4}{2-\alpha}\right)^2 -1.
\]
Let us now consider $T<+\infty$.\\
By using the generalization in Calculus of Variations of the classical 
Lagrangian multipliers method (cf. \cite{GF63} for further details) with respect to integral constraints 
(for us constraint \eqref{eq:infiniteconstraint})  and by virtue 
of the exponential decay of the variable $\rho$ stated in Equation \eqref{eq:stima_rho'/rho}, we can conclude that if $\gamma$ is a time-$T$ total collision solution for \eqref{eq:Newton} then $y$ is a (weak) extremal for the functional
\[
\mathbb{J} : W^{1,2}\left( [0,+\infty),\X\right) \times \R \to \R\cup\{+\infty\}
\]
given by
\begin{equation}\label{eq:actioniny}
\mathbb{J}(y,\lambda) := \int_0^{+\infty} \left\{ \dfrac{c_\alpha}{2}
(\Mnorm{y}')^2 + \dfrac12 \left[ \Mnorm{y'}^2 + 2\pothom(s)\Mnorm{y}^2\right] + 
\frac{\lambda}{\beta} (\Mnorm{y}^{\beta}-T) \right\}\, d\tau.
\end{equation}
In order to compute the Lagrange multiplier $\lambda$ we write the  
Euler-Lagrange equations associated to Equation \eqref{eq:actioniny}
\[
\lambda \Mnorm{y}^{\beta -2} y = y'' -(\alpha + 2)\Mnorm{y}^\alpha \pothom(y)y 
- \Mnorm{y}^{\alpha + 2}\nabla \pothom(y) + c_\alpha\left[ 
\frac{\Mnorm{y'}^2}{\Mnorm{y}^2} + \frac{y\cdot y''}{\Mnorm{y}^2} - 
\frac{(y\cdot y')^2}{\Mnorm{y}^4}\right]y,
\] 
and the expression of the first integral of the energy (in the variable $y$)
\begin{equation}\label{eq:energia}
h \= \frac{1}{\Mnorm{y}^\beta}\left\{ \frac{c_\alpha}{2} 
\left(\Mnorm{y}'\right)^2 + \frac12 \Mnorm{y'}^2 - 
\Mnorm{y}^{2+\alpha}\pothom(y) \right\}.
\end{equation}
We can now easily deduce that
\begin{equation}\label{eq:energy}
\frac{d}{d\tau}\left\{ \Mnorm{y}^{\beta}\left(\frac{\lambda}{\beta}-h\right)\right\} = 0
\end{equation}
hence necessarily 
\[
\lambda = \beta h.
\]
Finally we conclude that, given a time-$T$ total collision solution  $\gamma$ with energy $h$ then the corresponding $y$ is a (weak) extremal for
\begin{equation}\label{eq:actioniny_h}
\mathbb{J}(y) := \int_0^{+\infty} \left\{ \dfrac{c_\alpha}{2}
(\Mnorm{y}')^2 + \dfrac12 \left[ \Mnorm{y'}^2 + 2\pothom(s)\Mnorm{y}^2\right] + 
h(\Mnorm{y}^{\beta}-T) \right\}\, d\tau.
\end{equation}
Let us now consider $T=+\infty$.\\
In this case the function $y$ corresponding to the parabolic motion $\gamma$ does not belong anymore to $W^{1,2}\left( [0,+\infty);\X\right)$. 
Here we do not have to impose the integral constraints \eqref{eq:infiniteconstraint}, since $\|\rho^\beta\|_{L^1(0,+\infty)} = +\infty$ follows from the exponential growth of $\rho$. 
Nevertheless arguing precisely as in Lemma \ref{thm:azione-finita}, it turns out that $y$ is a weak $\mathscr{C}^{\infty}_0$-critical point for 
\begin{equation}\label{eq:actioniny_h=0}
\int_0^{+\infty} \left\{ \dfrac{c_\alpha}{2}
(\Mnorm{y}')^2 + \dfrac12 \left[ \Mnorm{y'}^2 + 2\pothom(s)\Mnorm{y}^2\right]\right\}\, d\tau,
\end{equation}
although such functional is not finite at $y$.

From now on, with a slight abuse of notation, both when $T$ is finite or $T=+\infty$, we will say that the trajectory $y(\tau) = \rho(\tau)s(\tau)$ is an asymptotic solution whenever it corresponds (through the variable changes given in Equations \eqref{eq:tau} and \eqref{eq:rho}) to an a.s. $\gamma(t)=r(t)s(t)$. 
The discussion carried out in this section ensures that an $s_0$-a.s. $y$ is a Gateaux $\mathscr{C}^{\infty}_0$-critical point for 
\begin{equation}\label{eq:Jfor-s_0am}
\mathbb{J}(y) = 
\begin{cases} \displaystyle 
\int_0^{+\infty} 
\dfrac{c_\alpha}{2}
(\Mnorm{y}')^2 + \dfrac12 \left[ \Mnorm{y'}^2 + 2\pothom(s_0)\Mnorm{y}^2\right] 
+ W(y) + 
h(\Mnorm{y}^{\beta}-T)\, d\tau, & \text{if } T<+\infty, \\
\displaystyle 
\int_0^{+\infty} 
\dfrac{c_\alpha}{2}
(\Mnorm{y}')^2 + \dfrac12 \left[ \Mnorm{y'}^2 + 2\pothom(s_0)\Mnorm{y}^2\right] 
+ W(y)\, d\tau, & \text{if } T=+\infty,
\end{cases}
\end{equation}
where
\begin{equation}\label{eq:W}
W(y) := \Mnorm{y}^2\left[ \pothom(s) - \pothom(s_0)\right] = 
\Mnorm{y}^{2+\alpha} \pothom(y) - \Mnorm{y}^2\pothom(s_0).
\end{equation}
In the special case $s(\tau)\equiv s_0$, that is when we consider an homothetic 
motion $y_0(\tau)=\rho_0(\tau)s_0$ (recall Remark \ref{rem:hom}), then $W(y_0)=0$.

%========================================
\subsection{Second Variation of the McGehee 
functional}\label{subsec:firstsecondvariation}
%========================================
The aim of this subsection is to compute the second variation and the 
associated quadratic form for the functional $\mathbb{J}$ defined in \eqref{eq:Jfor-s_0am}
along an $s_0$-a.s.. The following computation lemma will play an import role.
\begin{lem}\label{lem:secondW}
Let $s_0$ be a central configuration for $\pothom$ and $y_0(\tau) = \rho_0(\tau) s_0$ be an homothetic  motion 
(with a total collision or parabolic). Then the following relation holds
\[
d^2W(y_0)[\xi,\eta] = d^2\widetilde U(s_0)[\xi,\eta],
\qquad \forall\, \xi,\eta \in  \X,
\]
where $\widetilde U(x)\=\abs{x}_M^\alpha U(x)$, $x \in \X$.
\end{lem}
\begin{proof}
By a direct calculation, the $\xi$-directional derivative of the function $W$ 
defined in \eqref{eq:W} is given by 
\[
 \begin{split}
d W(y)[\xi] & = \dfrac{d}{d\varepsilon} 
\left[W(y+\varepsilon \xi)\right]\vert_{\varepsilon=0}\\
& = (\alpha+2)\Mnorm{y}^\alpha \Mprod{y}{\xi} \pothom(y)+ \Mnorm{y}^{\alpha+2} 
\langle\nabla \pothom(y), \xi\rangle-2 \Mprod{y}{\xi} \pothom(s_0).
\end{split}
\]
By a similar calculation we obtain 
\[
\begin{split}
d^2 W(y)[\xi,\eta] & =  \dfrac{d}{d\varepsilon} \left[W(y+\varepsilon 
\eta)[\xi]\right]_{\varepsilon=0} \\
& = 
\alpha(\alpha+2)\Mnorm{y}^{\alpha-2}\Mprod{y}{\xi}\Mprod{y}{\eta}\pothom(y)+ 
(\alpha+2) \Mnorm{y}^\alpha\Mprod{\xi}{\eta}
\pothom(y)\\
& + (\alpha+2)\Mnorm{y}^\alpha \Mprod{y}{\eta}\langle\nabla 
\pothom(y),\xi\rangle +(\alpha+2)\Mnorm{y}^\alpha\Mprod{y}{\xi} 
\langle\nabla \pothom(y), \eta\rangle\\
& + \Mnorm{y}^{\alpha+2} \langle D^2 \pothom(y)\xi,\eta \rangle -2 
\Mprod{\xi}{\eta}\pothom (s_0).
\end{split}
\]
Thus, since $\pothom$ is homogeneous of degree $-\alpha$, its first and second 
derivatives are respectively of  degree $-(\alpha+1)$ and $-(\alpha+2)$; choosing $y(\tau) = y_0(\tau) = |y_0(\tau)|s_0$ we 
infer
\[
\begin{split}
d^2 W(y_0)[\xi, \eta] = & \langle D^2 \pothom(s_0)\xi,\eta \rangle + 
(\alpha+2)\Mprod{s_0}{\xi}\langle\nabla \pothom(s_0), \eta\rangle + 
(\alpha+2)\Mprod{s_0}{\eta} \langle\nabla \pothom(s_0),\xi\rangle\\
& + \alpha(\alpha+2)\Mprod{s_0}{\eta}\Mprod{s_0}{\xi} \pothom(s_0) + 
(\alpha+2)\Mprod{\eta}{\xi}\pothom(s_0) -2 \Mprod{\xi}{\eta}\pothom (s_0).
\end{split}
\]
Being  $s_0$ a central configuration it holds $\nabla \pothom(s_0)=-\alpha 
\pothom(s_0) M s_0$ (cf. Definition \ref{def:cc}) and by the previous 
computation we immediately obtain
\[
d^2 W(y_0)[\xi, \eta] = \langle D^2 \pothom(s_0)\xi,\eta \rangle-\alpha 
(\alpha+2)\Mprod{s_0}{\eta}\Mprod{s_0}{\xi}
\pothom(s_0)+ \alpha  \Mprod{\xi}{\eta} \pothom (s_0).
\]
The thesis follows by taking into account that the last 
member in the previous formula is nothing but $d^2\widetilde U(s_0)[\xi,\eta]$ (for the details of this computation cf. \cite[Formula (8), Section 2]{BS08}).
\end{proof}
\begin{prop}\label{prop:var_seconda}
Let  $y = \rho s$ be an $s_0$-a.s. with energy $h$ (possibly 0). Then the 
quadratic form associated to the second variation of $\mathbb{J}$ at $y$ is
\begin{multline}\label{eq:secondavar}
d^2 \mathbb{J}(y)[\eta,\eta] = \int_0^{+\infty} 
\bigg\{ c_\alpha
        \bigg[ \left(\dfrac{\rho'}{\rho}\right)^2(
        \Mprod{s}{\eta}^2-\Mnorm{\eta}^2)
               +\Mprod{s'}{\eta}^2 +\Mprod{s}{\eta'}^2
       \\              +2\dfrac{\rho'}{\rho}\left[ \Mprod{\eta}{\eta'} 
-\Mprod{s}{\eta}\Mprod{s'}{\eta}
       -\Mprod{s}{\eta}\Mprod{s}{\eta'}\right] 
       + 2\Mprod{s'}{\eta}\Mprod{s}{\eta'}
        \bigg] 
        \\
        + \Mnorm{\eta'}^2 + 2\pothom(s_0)\Mnorm{\eta}^2 
+ d^2W(y)[\eta,\eta] +  h\beta\rho^{\beta-2}[(\beta-2)\Mprod{s}{\eta}^2 + \Mnorm{\eta}^2] \bigg\}\, d\tau
\end{multline}
for every  $\eta \in W^{1,2}_0([0,+\infty),\X)$.
\end{prop}
\begin{proof}
In order to prove the result we first af all compute the second G\^ateaux derivative of $\mathbb{J}$ in a direction $\eta \in \mathscr{C}^\infty_0
(0,+\infty;\X)$, hence, by density and using dominated convergence theorem (all terms in $\mathbb{J}$ are bounded) we compute it in any direction $\eta \in W^{1,2}_0([0,+\infty),\X)$.

The computation is straightforward 
although some care for the  first term in Formula \eqref{eq:actioniny}. Since 
\[
\Mnorm{y+\eps\eta} = \Mnorm{\rho s + \eps\eta}
=\sqrt{\rho^2 + 2\eps \rho \Mprod{s}{\eta} + \eps^2 \Mnorm{\eta}^2 },
\]
hence 
\[
\Mnorm{y+\eps\eta}'
= \dfrac{\rho\rho'+\eps \left(\rho'\Mprod{s}{\eta} + \rho \Mprod{s'}{\eta} + 
\rho \Mprod{s}{\eta'}\right) + \eps^2 \Mprod{\eta}{\eta'}}
{\sqrt{\rho^2 + 2\eps \rho \Mprod{s}{\eta} + \eps^2\Mnorm{\eta}^2}},
\]
and 
\[
\left(\Mnorm{y+\eps\eta}'\right)^2
= \dfrac{\left[\rho'+\eps \left(\dfrac{\rho'}{\rho} \Mprod{s}{\eta} + 
\Mprod{s'}{\eta} + \Mprod{s}{\eta'}\right) + 
\eps^2 \dfrac{\Mprod{\eta}{\eta'}}{\rho}\right]^2}{1 + 2\eps 
\dfrac{\Mprod{s}{\eta}}{\rho} + \eps^2\dfrac{\Mnorm{\eta}^2}{\rho^2} }.
\]
We now need to compute the second derivative of this expression with respect to 
$\eps$ and then make the limit as  $\eps \to 0$. We can then  use the McLaurin 
expansion of $(1+x)^{-1}$ as $x \to 0$,
obtaining (we just consider the terms containing $\eps$ up to the power 2),
\begin{multline*}
\left(\Mnorm{y+\eps\eta}'\right)^2 = 
\left[(\rho')^2 + 2\eps\rho'\left(\dfrac{\rho'}{\rho} \Mprod{s}{\eta} + 
\Mprod{s'}{\eta} + \Mprod{s}{\eta'}\right) \right. +\\
+ \left. \eps^2\left(\dfrac{\rho'}{\rho} \Mprod{s}{\eta} + \Mprod{s'}{\eta} + 
\Mprod{s}{\eta'}\right)^2\right. +\\
      + \left. 2\eps^2 \dfrac{\rho'}{\rho} \Mprod{\eta}{\eta'} + o(\eps^2) 
\right] 
\left[1 - 2\eps \dfrac{\Mprod{s}{\eta}}{\rho} + \eps^2\frac{4\Mprod{s}{\eta}^2 
- \Mnorm{\eta}^2}{\rho^2} + o(\eps^2) \right], \quad \eps \to 0.
\end{multline*}
When we compute the second derivative with respect to $\eps$ of the previous
expression and we evaluate it as $\eps=0$, the only term that will not vanish 
is 
the one with 
$\eps^2$; for this reason we just compute the coefficient of $\eps^2$ in the 
previous expression, that is
\[
\begin{split}
- \left(\dfrac{\rho'}{\rho}\right)^2\Mnorm{\eta}^2
+ 4\left(\dfrac{\rho'}{\rho}\right)^2\Mprod{s}{\eta}^2
+ \left(\dfrac{\rho'}{\rho} \Mprod{s}{\eta} + \Mprod{s'}{\eta} + 
\Mprod{s}{\eta'}\right)^2
+ 2\dfrac{\rho'}{\rho} \Mprod{\eta}{\eta'} \\
- 4\dfrac{\rho'}{\rho} \Mprod{s}{\eta} \left(\frac{\rho'}{\rho} \Mprod{s}{\eta} 
+ \Mprod{s'}{\eta} + \Mprod{s}{\eta'}\right) \\
= \Mprod{s'}{\eta}^2 + \Mprod{s}{\eta'}^2
        + 2\frac{\rho'}{\rho}\left[ 
\Mprod{\eta}{\eta'}-\Mprod{s}{\eta}\Mprod{s'}{\eta}-\Mprod{s}{\eta}\Mprod{s}{
\eta'} 
        \right]+\\ \left(\dfrac{\rho'}{\rho}\right)^2(\Mprod{s}{\eta}^2
-\Mnorm{\eta}^2)
        +2\Mprod{s'}{\eta}\Mprod{s}{\eta'}.
\end{split}
\]
We are now able to the deduce the expression of the thesis.
\end{proof}
The next result follows straightforwardly from Remark \ref{rem:hom_rho}, Lemma \ref{lem:secondW} and Proposition \ref{prop:var_seconda}.
{
\begin{cor}\label{cor:secondvarhom}
Let $y^*(\tau)=\rho_0(\tau)s_0$ be an $s_0$-homothetic parabolic motion. Then the quadratic form associated to the second variation of $\mathbb{J}$ at $y^*$ is 
\begin{equation}\label{eq:secondavarhom_eta}
\begin{split}
d^2 \mathbb{J}(y^*)[\eta,\eta] =  \int_0^{+\infty} 
& \Big\{ c_\alpha
        \Big[ \overline\delta_\alpha^2(\Mprod{s_0}{\eta}^2-\Mnorm{\eta}^2)
        +\Mprod{s_0}{\eta'}^2 + 2\overline\delta_\alpha \left( \Mprod{\eta}{\eta'} 
\right. \\
&\left. -\Mprod{s_0}{\eta}\Mprod{s_0}{\eta'}\right) \Big]
        + \Mnorm{\eta'}^2 + 2\pothom(s_0)\Mnorm{\eta}^2 + d^2\widetilde U(s_0)[\eta,\eta] \Big\}\, d\tau
\end{split}
\end{equation}
for every  $\eta \in W^{1,2}_0([0,+\infty),\X)$.
\end{cor}
}

%========================================
%========================================
\section{Spectral index and geometrical index for  asymptotic motions}
\label{sec:fredholmness}
%========================================
%========================================

This section splits into two subsections. In Subsection 
\ref{subsec:spectral-index} we associate to 
an $s_0$-asymptotic solution a Morse-type index that we shall refer to as the {\em spectral index\/} whilst in Subsection \ref{subsec:geometrical-index} we introduce the second main 
object of our theory which is the {\em geometrical index\/} defined through the Maslov index. 

%========================================
\subsection{The spectral index}\label{subsec:spectral-index}
%========================================
%
We denote by $u\otimesm v$ the tensor product of the two vectors $u$ and $v$ 
made with respect to the mass scalar product, i.e. for any vector $\eta$, we 
obtain $(u \otimesm v)\eta= \Mprod{v}{\eta}u$.
We also observe  that $\trasp{(u \otimesm v)}=(v \otimesm u)$.

Given an $s_0$-asymptotic solution $y$, let us introduce the endomorphisms
\begin{equation}\label{eq:matrici}
\begin{split}
 P(\tau) & \= c_\alpha s \otimesm s+ I  \\
 Q(\tau) & \= c_\alpha \left(\dfrac{\rho'}{\rho}\right)( I - s \otimesm s) + 
c_\alpha s\otimesm s',
\\
 R(\tau) & \=  \left[c_\alpha \left(\dfrac{\rho'}{\rho}\right)^2 + 
\beta(\beta-2) h \rho^{\beta-2} \right] s \otimesm s + \left[ 2 \pothom(s_0)  - 
c_\alpha \left(\dfrac{\rho'}{\rho}\right)^2+ 
\beta h \rho^{\beta-2} \right] I \\ 
& + c_\alpha s'\otimesm s' - c_\alpha \left(\dfrac{\rho'}{\rho}\right) 
\left(s\otimesm s' + s'\otimesm s\right).
\end{split}
\end{equation}
\begin{defn}\label{def:index_quadr_AS}
Given an $s_0$-a.s. $y$, we define the {\em $s_0$-asymptotic index form of $y$}, 
$\IAS$, and the {\em $s_0$-asymptotic quadratic form of $y$\/}, $\qAS$, 
respectively as the bilinear and the quadratic forms associated to $d^2 
\mathbb{J}(y)$ (computed in Proposition \ref{prop:var_seconda}), namely
\begin{equation}
\label{eq:IAS}
\IAS(\xi,\eta) :=
\int_0^{+\infty}\left[
\Mprod{ P(\tau)\xi'}{\eta'}+ 
\Mprod{ Q(\tau)\xi}{\eta'}+
\Mprod{\trasp{Q}(\tau)\xi'}{\eta} + \Mprod{\widetilde { R}(\tau)\xi}{\eta}
\right]\, d\tau,
\end{equation}
and
\begin{equation}\label{eq:qAS}
\qAS(\eta) \= \IAS(\eta,\eta),
\end{equation}
for any $\xi,\eta \in W^{1,2}_0\left([0,+\infty),\X\right)$, where 
\[
\widetilde {R}(\tau) =  R(\tau) + D^2W(y).
\]
\end{defn}
{
When $y$ is an $s_0$-homothetic motion $s(\tau)=s_0$ for any $\tau$ so that the matrix $P$ reduces to a constant matrix, while $Q$ and $R$ strongly simplify (since $s'=0$). When we compute these matrices for an $s_0$-homothetic parabolic motion then, since $\dfrac{\rho'(\tau)}{\rho(\tau)} = \bar{\delta}_{\alpha}$ (cf. Remark \ref{rem:hom_rho}) and $h=0$, we obtain
\begin{equation}\label{eq:matrici_0}
\begin{split}
P_{0} &\= c_\alpha s_0 \otimesm s_0+I, \\
Q_{0} &\= c_\alpha\bar\delta_\alpha( I -s_0 \otimesm s_0)\\
R_{0} &\=  c_\alpha \bar \delta_\alpha^2 s_0 \otimesm  s_0 + \left( 2 \pothom(s_0)  - c_\alpha \bar \delta_\alpha^2\right)I \\
&= c_\alpha \bar \delta_\alpha^2  s_0 \otimesm  s_0 +  
\frac{(2-\alpha)^2}{8}\pothom(s_0)  I. 
\end{split}
\end{equation}
By virtue of Corollary \ref{cor:secondvarhom} and by performing 
an integration by part ($Q_0$ is now a constant matrix), we are allowed to give the following definition.
\begin{defn}\label{def:index_quadr_lim}
	Given an $s_0$-homothetic parabolic motion, we  define the {\em $s_0$-limit index form\/}, $\IAS_0$, and the {\em  $s_0$-limit quadratic form\/}, $\qAS_0$, respectively as the bilinear and the quadratic form associated to $d^2 \mathbb{J}(y_0)$ (computed in Corollary \ref{cor:secondvarhom}), namely
	\begin{equation}
	\label{eq:secondavar2_0}%\label{eq:limit-index-forms}
	\Ilim (\xi,\eta):= \int_0^{+\infty} 
	\left[\Mprod{ P_0 \xi'}{\eta'}+ \Mprod{\widetilde{R_0}\xi}{\eta} 
	\right] \, d\tau,
	\end{equation}
	and
	\[
	\qlim(\eta) := \Ilim(\eta,\eta),
	\]
	for any $\xi,\eta \in W^{1,2}_0\left([0,+\infty),\X\right)$,
	where
	\begin{equation}\label{eq:R0tilde}
	\widetilde {R_0} \= R_{0}+{D^2\widetilde U(s_0)}
	\end{equation} 
\end{defn}
We denote by $\Sym{k}$ the space of all $k\times k$ symmetric matrices.  
Let $D:[0,+\infty) \to \Sym{2N}$ be the continuous path pointwise defined by 
\begin{equation}\label{eq:D}
D(\tau)\=\begin{bmatrix}
P(\tau) &  Q(\tau)\\   
\trasp{ Q}(\tau)&  R(\tau)
\end{bmatrix}.
\end{equation}
Next lemma describes the asymptotic behaviour of the matrices $R(\tau)$, $Q(\tau)$ and $P(\tau)$.
\begin{lem}\label{lem:conv_matrici}
	For every $\varepsilon >0$, there exists $T >0$ such that 
	\[
	\sup_{\tau \in [T, +\infty)}|D(\tau)-D_0| \leq \varepsilon,
	\]
	where
	\begin{equation}\label{eq:limiting}
	D_0\=\begin{bmatrix}
	P_0 &  Q_0 \\  \trasp{Q}_0 &  R_0
	\end{bmatrix} \in \Sym{2N}
	\end{equation}
	{and where $|\cdot|$ denotes the norm induced by $M$.}
\end{lem}
\begin{proof}
	The proof of this result readily follows by taking into account the limits in Equations \eqref{eq:stima_rho'/rho} and \eqref{eq:stima_s'}, the assumption that $s(\tau)$ tends to $s_0$, and the fact that $D^2W(y(\tau))$ converges to $D^2W(y_0) = D^2\widetilde \pot(s_0)$ (by Lemma \ref{lem:secondW}), as $\tau \to +\infty$.
\end{proof}
%
%%{\color{red} ATTENZIONE Mi pare che non serva studiare lo spettro di $R(\tau)$ per moti omotetici con energia non nulla: si fa facilemente aggiungendo $o(1)$ agli autovalori di $R_0$. Nel caso di moti omotetici, ma con energia non nulla, per dare un segno alla forma quadratica bisognerebbe\\
%%	- sapere in ch modo $\rho'/\rho$ tende alla costante \\
%%	- dare un segno alla matrice $Q$, non costante
%%MI pare quindi che non si riesca, seppur si possa studiare lo spettro di $R(\tau)$.
%%}

Next result gives some useful spectral properties of the matrices $P_0$ and $ R_0$.
\begin{lem}\label{lem:spectra}
The matrices $P_0$, $R_0$ and $ P(\tau)$, $\tau \in [0,+\infty)$, are positive definite and their spectrum is  
\[
\mathfrak{sp}( P_0)=\mathfrak{sp}( P(\tau)) = \{1, 1+c_\alpha\} \qquad \text{for any } \tau \in [0,+\infty),
\]
and 
\[
\mathfrak{sp}( R_{0}) = \left\{ r^1_{0},r^2_{0}\right\}
\]
where 
\[
r^1_{0} \= 2\pothom(s_0) - c_\alpha \bar\delta_\alpha^2 
= \dfrac{(2-\alpha)^2}{8}\pothom(s_0)
\;\text{ and }\;
r^2_{0} \= 2\pothom(s_0) = \dfrac{(2-\alpha)^2}{8}\pothom(s_0) 
+ c_\alpha \bar\delta_\alpha^2.
\]
Furthermore, the eigenspace corresponding to $1 \in \mathfrak{sp}( P_0)$ and the one correspondig to $r^1_{0} \in \mathfrak{sp}(R_0)$ is $s_0^{\perp_M}$, 
while the one corresponding to $1+c_\alpha \in \mathfrak{sp}( P_0)$ and to $r^2_{0} \in \mathfrak{sp}(R_0)$ is the vector space spanned by $s_0$.
\end{lem}
\begin{proof}
For any vector $w$, $w\otimesm w$ is a rank 1 symmetric matrix, indeed, for any vector $v$
\begin{equation}\label{eq:contofava}
(w \otimesm w)v = \Mprod{w}{v} w,
\end{equation}
moreover the kernel of $w \otimesm w$ is $\{v:\Mprod{v}{w}=0\}$ and the only non vanishing eigenvalue is $\Mnorm{w}^2$, indeed 
\[
(w \otimesm w)w = \Mprod{w}{w} w = \Mnorm{w}^2 w.
\]
It is readily seen that the spectrum of $P(\tau)$ does not depend on $\tau$ and is $\{1, 1+c_\alpha\}$; we immediately conclude that $P(\tau)$ is positive definite and that the corresponding eigenspaces are $s_0^{\perp_M}$ and $\R s_0$. 

In order to determine the spectrum of $R_{0}$, we simply recall that 
$c_\alpha = 
\left(\frac{4}{2-\alpha}\right)^2-1$ and $\bar \delta_\alpha^2 = 
\frac{(2-\alpha)^2}{8}\pothom(s_0)$, so that
\[
2\pothom(s_0) - c_\alpha \bar \delta_\alpha^2 = 
\frac{(2-\alpha)^2}{8}\pothom(s_0).
\]
The conclusion follows by arguing in the very same way as before. This conclude the proof. 
\end{proof}
The following condition will play a central role in the rest of the paper:
\begin{center}
{\bf \BS : \, $\widetilde {R_0}$ is positive definite.\/}
\end{center}
Let us define 
\[
\mu_1 \= \text{the smallest eigenvalue of } {D^2\widetilde U(s_0)}.
\]
By taking into account Lemma \ref{lem:spectra}, the definition of the matrix $\widetilde{R_0}$, and the fact that the kernel of ${D^2\widetilde U(s_0)}$ coincides with $\R s_0$, we have that
\begin{center}
\BS \quad $\Longleftrightarrow$ \quad $\mu_1 +\dfrac{(2-\alpha)^2}8{\pot(s_0)} >0.$ 
\end{center}
}
In fact, the two eigenvalues of the endomorphism $R_0$ are ordered, being 
\[
r^1_{0}=\dfrac{(2-\alpha)^2}{8}\pothom(s_0)
\leq 
r^2_{0} = \dfrac{(2-\alpha)^2}{8}\pothom(s_0) 
+ c_\alpha \bar\delta_\alpha^2.
\]
Thus, in particular, if $\mu_1 + r_0^1 >0$ then also $\mu_1+ r_0^2$ is positive and thus $\widetilde R_0$ is positive definite.
{
\begin{lem}\label{lem:ess-pos-limit}
Let {$\qlim$} be the quadratic form associated to an $s_0$-homothetic parabolic motion. Then: 
\begin{enumerate}
	\item $\qlim$ is an equivalent norm in $W^{1,2}_0([0,+\infty);\X)$ if and only if $s_0$ satisfies \BS;
	\item $\qlim(\eta) \geq 0$ for any  $\eta \in W^{1,2}_0([0,+\infty);\X)$ if and only if $\mu_1 \geq -\frac{(2-\alpha)^2}{8}U(s_0)$;
\end{enumerate}
\end{lem}
\begin{proof}
We start by proving claim \emph{(1)}.\\
Let us assume that the inequality \BS\, holds; then, by virtue of Lemma \ref{lem:spectra}, there exists $\delta>0$ such that for any $\eta \in W^{1,2}_0([0,+\infty);\X)$
\[
\int_0^{+\infty} \Mprod{ P_0 \eta'}{\eta'} d\tau
\geq \norm{\eta'}_{L^{2}(0,+\infty)}^2
\quad \text{and} \quad 
\int_0^{+\infty} 
\Mprod{\widetilde R_0\eta}{\eta}d\tau \geq \delta \norm{\eta}_{L^{2}(0,+\infty)}^2,
\]
hence
\begin{equation}\label{eq:1--1}
\qlim (\eta) \geq 
\min(1,\delta) \norm{\eta}_{W^{1,2}_0([0,+\infty))}^2, 
\qquad \forall \eta \in W^{1,2}_0([0,+\infty);\X).
\end{equation}
In order to prove the {\em only if} part we argue by contradiction, assuming that 
\[
\mu_1 \leq  -\frac{(2-\alpha)^2}{8} U(s_0).
\]
Let $w_1$ be such that $\Mprod{\widetilde R_0 w_1}{w_1} = \left( \mu_1 + \frac{(2-\alpha)^2}{8} U(s_0)\right) \Mnorm{w_1}^2$.
First of all we assume that the strict inequality holds; then, by taking advantage of the presence of the eigendirection $w_1$,
we can show (arguing as in the proof of \cite[Theorem 4.3]{BS08}) the existence of infinitely many $\xi_n$ such that {$\qlim(\xi_n)\leq 0$}. Hence $\qlim$ is not a norm. \\
Finally, if the equality 
$\mu_1 = -\frac{(2-\alpha)^2}{8} U(s_0)$ holds, then for any $\xi \in W^{1,2}_0([0,+\infty);\X)$ such that $\xi(\tau)/\Mnorm{\xi(\tau)}={w_1}$ for any $\tau$, there holds 
\[
\qlim(\xi) = \int_0^{+\infty}\Mnorm{\xi'}^2
\]
which means that $\qlim$ is not an equivalent norm in $W^{1,2}_0([0,+\infty);\X)$. \\
The proof of claim \emph{(2)} is substantially the same. In fact, if the inequality on the first eigenvalue of $D^2\widetilde U(s_0)$ holds, then $\qlim(\eta)\geq 0$ for any $\eta \in W^{1,2}_0([0,+\infty);\X)$. If the equality holds, then the smallest eigenvalue of the matrix $\widetilde R_0$ vanishes, so that $\qlim(\eta) = \int_0^{+\infty}\Mnorm{\eta'}^2 \geq 0$, for any $\eta \in W^{1,2}_0([0,+\infty);\X)$. In order to prove the \emph{only if} part we argue by contradiction assuming that $\mu_1 < -\frac{(2-\alpha)^2}{8} U(s_0)$ and showing the existence of infinitely many $\xi_n$ such that {$\qlim(\xi_n)\leq 0$}.
\end{proof}
In the next definition we introduce the {\em spectral index} of an $s_0$-a.s. $y$: this is nothing but the Morse index of 
the motion $y$, seen as a critical point (in the smooth sense described in Lemmata \ref{le:weak_sol} and \ref{thm:azione-finita}) of the functional $\mathbb{J}$ introduced at page 
\pageref{eq:W}; for this reason we will often  refer to it as the {\em Morse index of $y$\/}. 
\begin{defn}\label{def:spectral-index}
Let $s_0$ be a central configuration and $y$ be an $s_0$-a.s.\,. We  define the {\em spectral index\/ of $y$} as 
\[
\ispec(y) \= \dim E_-(\qAS) \in \N \cup \{+\infty\},
\]
where $E_-(\qAS)$ denotes the negative spectral space of one (and hence all) 
realisation of $\qAS$ (cf. Appendix \ref{sec:spectral-flow-and-Maslov} for 
further details on the well-posedness of this number).
\end{defn}
\begin{rem}
Let us observe that the variable changes introduced in Subsection \ref{subsec:McGehee} ensures that, when $T=+\infty$ the Morse index of $y$ coincides with the Morse index of the corresponding a.s. $\gamma$ (seen as a critical point on $\SSS$). When $T<+\infty$, the difference between these two indices is at most one: this is due to the presence of the 1-codimensional  constraint given in Equation \eqref{eq:infiniteconstraint}.
\end{rem}
Given an $s_0$-homothetic parabolic motion $y^*$, in Lemma \ref{lem:spectra} we finally proved that 
\[
\ispec(y^*) < +\infty \; \text{(and indeed vanishes)}
\qquad \text{if and only if} \qquad
\mu_1 \geq -\frac{(2-\alpha)^2}{8}U(s_0).
\]

In the next result we deal with a generic $s_0$-a.s. and we prove that \BS-condition is indeed sufficient for the finiteness of its spectral index.
\begin{thm}\label{prop:key1-una freccia}
Let $s_0$ be a central configuration and let $y$ be an $s_0$-a.s..
If $s_0$ satisfies the \BS-condition then
\[
\ispec(y) < +\infty.
\]
\end{thm}
\begin{proof}
Let $\qAS$ be the quadratic form associated to the asymptotic motion $y$ and let us assume that \BS-condition holds. By virtue of the spectral property of $P(\tau)$ (cf. Lemma \ref{lem:spectra}) we have, for any $\eta \in W^{1,2}_0([0,+\infty);\X)$,
\begin{multline}\label{eq:decompQ}
\qAS(\eta)  \geq 
\norm{\eta'}^2_{L^2(0,+\infty)} 
+ \int_0^{+\infty} 
\Big[ \Mprod{Q(\tau)\eta}{\eta'}
+\Mprod{Q^T(\tau)\eta'}{\eta} +
\Mprod{\widetilde R(\tau)\eta}{\eta} 
\Big]\,d\tau \\
 =
\norm{\eta'}^2_{L^2([0,+\infty))} 
+ \int_0^{+\infty} 
\Mprod{\widetilde R_0\eta}{\eta} \,d\tau\\
+ \int_0^{+\infty} 
\Big[ \Mprod{(Q(\tau)-Q_0)\eta}{\eta'}
+\Mprod{(Q^T(\tau)-Q^T_0)\eta'}{\eta} +
\Mprod{(\widetilde R(\tau)-\widetilde R_0)\eta}{\eta} 
\Big]\,d\tau.
\end{multline}
On the one hand, by virtue of Lemma \ref{lem:ess-pos-limit} \emph{(1)} there exists $\epsilon \in (0,1)$ such that
\begin{equation}\label{eq:stima1}
\int_0^{+\infty} 
\Mprod{\widetilde R_0\eta}{\eta} \,d\tau \geq
\epsilon \norm{\eta}^2_{L^{2}(0,+\infty)};
\end{equation}
on the other hand Lemma \ref{lem:conv_matrici} ensures that for any $\delta >0$ there exists  $T_\delta > 0$ such that, for any $\eta \in W^{1,2}_0([0,+\infty);\X)$,
\begin{multline}\label{eq:stima2}
\int_{T_\delta}^{+\infty}
\Big[
\Mprod{(Q(\tau)-Q_0)\eta}{\eta'}
+\Mprod{(Q^T(\tau)-Q^T_0)\eta'}{\eta} +
\Mprod{(\widetilde R(\tau)-\widetilde R_0)\eta}{\eta} 
\Big]\,d\tau \\
\geq -\delta \norm{\eta}^2_{L^2(T_{\delta},+\infty)} 
\geq -\delta \norm{\eta}^2_{L^2(0,+\infty)}.
\end{multline}
Let us choose $\delta = \frac{\epsilon}{2}$ and term $T:=T_{\frac{\epsilon}{2}}$. Replacing Equations \eqref{eq:stima1} and \eqref{eq:stima2} in \eqref{eq:decompQ} we obtain
\begin{multline}\label{eq:decompQ2}
\qAS(\eta) \geq \norm{\eta'}^2_{L^2(0,+\infty)} 
+ \frac{\epsilon}{2} \norm{\eta}^2_{L^{2}(0,+\infty)} 
\\
+ \int_0^{T}
\Big[
\Mprod{(Q(\tau)-Q_0)\eta}{\eta'}
+\Mprod{(Q^T(\tau)-Q^T_0)\eta'}{\eta} +
\Mprod{(\widetilde R(\tau)-\widetilde R_0)\eta}{\eta} 
\Big]\,d\tau.
\end{multline}
Let us now estimate separately the terms in the last integral. Since we integrate on a compact interval, for some positive constants $c_1$ and $c_2$, and for any $s>0$, there holds 
\[
\int_0^{T}
\Mprod{(\widetilde R(\tau)-\widetilde R_0)\eta}{\eta} \,d\tau \geq -c_1\norm{\eta}^2_{L^2(0,T)}
\]
and 
\[
\int_0^{T}
\Mprod{(Q(\tau) - Q_0)\eta}{\eta'} \,d\tau \geq -c_2\int_0^{T}
\Mprod{\eta}{\eta'} \,d\tau
\geq -\frac{c_2}{4s}\norm{\eta}^2_{L^2(0,T)}
-{c_2}{s}\norm{\eta'}^2_{L^2(0,T)},
\]
where in the last step we used the inequality
$\langle a,b \rangle \leq \frac{\norm{a}^2}{4s} +s\norm{b}^2$, which holds in an Hilbert space for any $s>0$.\\
The integral term in Equation \eqref{eq:decompQ2} can be estimated from below by the sum
\[
-\frac{c_2}{2s}\norm{\eta}^2_{L^2(0,T)}
-2sc_2 \norm{\eta'}^2_{L^2(0,T)}
-c_1 \norm{\eta}^2_{L^2(0,T)}
\]
so that 
\begin{multline*}
\qAS(\eta) \geq \norm{\eta'}^2_{L^2(0,+\infty)} 
+ \frac{\epsilon}{2} \norm{\eta}^2_{L^{2}(0,+\infty)} 
-\left(c_1+\frac{c_2}{2s}\right)\norm{\eta}^2_{L^2(0,T)} - 2sc_2 \norm{\eta'}^2_{L^2(0,T)} \\
\geq 
(1-2sc_2)\norm{\eta'}^2_{L^2(0,+\infty)} 
+ \frac{\epsilon}{2} \norm{\eta}^2_{L^{2}(0,+\infty)} 
-\left(c_1+\frac{c_2}{2s}\right)\norm{\eta}^2_{L^2(0,T)}.
\end{multline*}
Choosing $s=\frac{1}{2c_2}\left(1-\frac{\epsilon}{2}\right)$ we claim the existence of a positive constant $k$ such that
\[
\qAS(\eta) \geq \frac{\epsilon}{2}\norm{\eta}^2_{W^{1,2}_0([0,+\infty))} 
-k\norm{\eta}^2_{L^2(0,T)}.
\]
Let us now recall that $W^{1,2}_0([0,+\infty);\X)$ is continuously embedded in $W^{1,2}([0,T];\X)$ which is compactly embedded in $L^2(0,T;\X)$.
Hence $W^{1,2}_0([0,+\infty);\X)$ is compactly embedded in $L^2(0,T;\X)$. The (RHS) is then a compact perturbation of a positive quadratic form (the norm), hence its spectral index is finite. By monotonicity of the spectral index, we conclude that also the spectral index of the form $\qlim$ is finite. 
\end{proof}
\begin{rem}\label{rem:vi}
In \cite[Definition 4.2]{BS08} the authors introduced, dealing with total collision trajectories admitting an asymptotic direction, the {\em collision Morse index\/}: this index is nothing but the spectral index computed with respect to angular variations, namely  with respect to the variations pointwise $M$-orthogonal to $s_0$. The main result of the aforementioned paper states that, given an $s_0$-a.s., then
\[
	\mu_1<  -\dfrac{(2-\alpha)^2}8{\pot(s_0)}  
	\qquad \implies \qquad
	\ispec(y) = +\infty.
\]
Combining this result with the finiteness of the spectral index under [BS]-condition we claim that the  main theorem in \cite{BS08} is sharp, understanding that in the limiting case
$
\mu_1=  -\dfrac{(2-\alpha)^2}8{\pot(s_0)}
$
the spectral index could be finite or not.

Furthermore we infer that not only $s_0$-a.s. to a minimal central configuration but also $s_0$-a.s. to a saddle one with a ``not so negative minimal eigenvalue'' has  finite Morse index. Hence these asymptotic motions may be responsible of a change in the topology of the sub-levels of the Lagrangian action and actually plays a role in the topological balance. All others $s_0$-a.s. are  somehow ``{\em invisible\/}''.
\end{rem}
}
\begin{ex}[{\em Collinear central configuration, $N=3$\/}]
Let $N=3$ and $s_0$ be the collinear central configuration. Assume that 
$m_1=m_3=1$, $m_2=m\geq 1$ and that the mass $m_2$ stays in between the other 
two.   
Following the computations in \cite[Section 5]{BS08} we can assert that if $m 
\leq M 
\approx 3.51$ then 
there exists $\bar \alpha \in [0,1)$ such that for any $\alpha \in [\bar 
\alpha, 
2)$ the \BS-condition is not satisfied. 

When we focus on the Keplerian case, fixing $\alpha = 1$, in \cite{Hua11} we 
find the optimal condition on the mass $m$ in order to obtain a finite spectral 
index for any $s_0$-a.s., that is
\[
s_0 \text{ satisfies  \BS } 
\quad \Leftrightarrow \quad
m > \dfrac{27}{4}.
\]
\end{ex}

%========================================
\subsection{The geometrical index}\label{subsec:geometrical-index}
%========================================
%
This Section is devoted to set the geometrical framework behind the construction 
of a {\em geometrical index\/} associated to a $s_0$-a.s., through a Maslov-type 
index defined as intersection index in the Lagrangian Grassmannian manifold  
(for further details we refer the reader to Appendix 
\ref{sec:spectral-flow-and-Maslov} and to 
\cite{HP17} and references therein). 

Before introducing the symplectic setting, we recall that the configuration space  $\X$ is a $N=d(n-1)$-dimensional linear subspace of $\R^{nd}$; let us denote with $\phi: \R^N \to \X$ a linear isomorphism. Now, any $N\times N$ endomorphism, $A:\X \to \X$, defined in the previous subsection (we refer to $P(\tau)$, $Q(\tau)$, $R(\tau)$, $P_0$, $Q_0$, $R_0$, \ldots), can be identified with 
\[
\trasp{\phi}\,M\, A\, \phi : \R^N \to \R^N.
\]
We keep the same names to identify such endomorphisms.

%% Being the dual a co-functor, $\phi$ induces a unitary map (just by lifting on the cotangent), $T^*\phi: T^* \X \to T^* \R^N$. Thus, we can endow $T^* \X$ with the symplectic structure $\omega$ obtained by pulling-back, through $T^*\phi$, the standard symplectic structure of $\R^N$.  

In order to introduce the geometrical index in terms of the
Maslov index, we start by considering the Hamiltonian systems
arising from the index forms introduced in Section \ref{sec:fredholmness}
(cf. Definitions \ref{def:index_quadr_AS} and \ref{def:index_quadr_lim}). 

With this aim given an $s_0$-a.s. $y$, we define the {\em $s_0$-asymptotic Sturm-Liouville operator\/} as the closed selfadjoint linear operator $
\mathcal T:\mathscr D(\mathcal T)\subset L^2([0,+\infty), \R^N) \to  L^2([0,+\infty), \R^N) $
defined by 
\begin{equation}\label{eq:sturm-liouville}
\mathcal T u\=-\dfrac{d}{d\tau}\Big(P(\tau) u' + Q(\tau) u \Big) + 
\trasp{Q}(\tau)u' +   \widetilde R(\tau) u, 
\qquad \tau \in (0,+\infty)
\end{equation}
having dense domain $\mathscr D(\mathcal T)\= W^{2,2}([0,+\infty), \R^N)\cap W^{1,2}_0([0,+\infty), \R^N)$. \\
Finally, we term {\em $s_0$-limit Sturm-Liouville operator\/}, the  
operator  defined by 
\begin{equation}\label{eq:sturm-liouville-limit}
\mathcal T^* u\=-\dfrac{d}{d\tau}\Big(P_0 u' + Q_0 u \Big) + 
\trasp{Q}_0u' +   \widetilde R^*_0 u, 
\qquad \tau \in (0,+\infty).
\end{equation}

We now associate to each Sturm-Liouville operator the corresponding 
Hamiltonian system. More precisely, $u \in \ker({\mathcal T})$ iff $z$, pointwise defined by $z(\tau)=\big(P(\tau) u'(\tau) + Q(\tau)u(\tau), u(\tau)\big)$, is a solution of the 
Hamiltonian system 
\begin{equation}\label{eq:Ham-L2-coll}
z'(\tau)= H(\tau)\, z(\tau),
\qquad \tau \in [0,+\infty)
\end{equation}
where $H(\tau)\=J B(\tau)$ with
\begin{equation}\label{eq:H-coll-1}
B(\tau)\= \begin{bmatrix}
  P^{-1}(\tau) & -P^{-1}(\tau)Q(\tau)\\
  -\trasp{Q}(\tau)P^{-1}(\tau) & \trasp{Q}(\tau)P^{-1}(\tau)Q(\tau)-\widetilde R(\tau) 
\end{bmatrix}.            
\end{equation}
The  system defined in Equation \eqref{eq:Ham-L2-coll} will be termed 
{\em  $s_0$-aymptotic Hamiltonian system\/}. Analogously, to the Sturm-Liouville 
operator $\mathcal T^*$ we associate the first order Hamiltonian System 
termed {\em 
$s_0$-limit Hamiltonian system\/} and defined as follows 
\begin{equation}\label{eq:Ham-L2-coll-limit-nuovo}
z'(\tau)= H^*\, z(\tau)
\qquad \textrm{ for }\tau \in [0,+\infty)
\end{equation}
where $H^* \=J  B^*$ and  
\begin{equation}
\label{eq:H-coll-1-limit-nuovo} 
B^* \= \begin{bmatrix}
  P^{-1}_0 & -P_0^{-1}Q_0\\
         -\trasp{Q}_0P_0^{-1}&  \trasp{Q}_0P_0^{-1} Q_0-\widetilde{R_0}
         \end{bmatrix}
\end{equation}
%%{\color{red}{
%%\begin{oss}
%%It is worth noticing that it is possible equip $T^*\X$ with the {\em mass symplectic structure\/}; namely the symplectic structure induced at cotangent level by the mass scalar product:
%%\[
%%\omega_M(\cdot, \cdot)=\langle J \cdot, \cdot \rangle_{\Delta_M}
%%\]
%%where $\Delta_M\= \begin{bmatrix}
%% M^{-1}& 0 \\
%% 0 & M
%% \end{bmatrix}$. By general result there exists a symplectic basis such that $J$ is the standard complex structure.
%% \end{oss}
%%  }}
Next result is a technical lemma that aims to simplify the expression of the matrix $H^*$ in order to prove its hyperbolicity, in Proposition \ref{prop:bs-iperb}.
\begin{lem}\label{thm:tecnico}
We let 
\[
Z:=
\begin{bmatrix}
\Id & A\\
\trasp{A} & \trasp{A}A-C
\end{bmatrix}
\]
where $A \in \Sym{N}$ and $C$ is a $N$-dimensional matrix. 
Then $JZ$ is symplectic similar to $J\widehat{Z}$ where 
\[
\widehat{Z}:= 
\begin{bmatrix}
\Id & 0\\
0 & -C
\end{bmatrix}.
 \]
\end{lem}
\begin{proof}
In order to prove the result, we define the symplectic matrix 
$\displaystyle 
P\= 
\begin{bmatrix}
 \Id & -A\\ 
 0 & \Id
\end{bmatrix}.$  By a direct calculation, the following equality holds 
$\displaystyle 
\trasp{P} Z P = \widehat{Z}$ 
hence $\displaystyle 
J \widehat{Z} = J \trasp{P} Z P= {P}^{-1}(J Z)P, $ 
and the thesis readily follows.
\end{proof}
\begin{lem}\label{thm:tecnico2}
The Hamiltonian matrices $JT^*$ and $J \widehat{T^*}$ where
\[
T^*  \= \begin{bmatrix}
\Id & -c_\alpha \overline \delta_\alpha \Id\\
-c_\alpha \overline \delta_\alpha \Id & (c_\alpha^2\overline \delta_\alpha)^2 \Id -
\dfrac{(2-\alpha)^2}{8} \pot(s_0)\Id -{D^2\widetilde U(s_0)}
% D^2 \pot|_{\mathcal E}(s_0)
\end{bmatrix},
\]
and 
\[
\widehat{T^*}  \= \begin{bmatrix}
\Id & 0\\
0   & -\dfrac{(2-\alpha)^2}{8} \pot(s_0)\Id - {D^2\widetilde U(s_0)}
%D^2 \pot|_{\mathcal E}(s_0)
\end{bmatrix}
\]
are symplectically similar.
\end{lem}
\begin{proof}
The proof of this result readily follows by Lemma \ref{thm:tecnico} 
once setting 
\[
A:= -c_\alpha \overline \delta_\alpha \Id
\qquad \text{and} \qquad 
C:= \dfrac{(2-\alpha)^2}{8} \pot(s_0) \Id +{D^2\widetilde U(s_0)}
%+ D^2 \pot|_{\mathcal E}(s_0)
\]
\end{proof}
Let $E,F$ be two Euclidean vector spaces, $A \in \Sp(E \times E)$ and $B \in \Sp(F \times F)$; the diamond product $A \diamond B$ is a symplectic endomorphism of $(E \oplus F) \times (E \oplus F)$ whose corresponding matrix (once a symplectic basis is chosen) is the diamond product between the matrices corresponding to $A$ and $B$ respectively (cf. \cite{Lon02}, for further details). 
\begin{prop}\label{prop:bs-iperb}
The \BS-condition is equivalent to the  hyperbolicity of the  Hamiltonian endomorphism $H^* = JB^*$.
\end{prop}
\begin{proof}
We start to observe that the symmetric endomorphism $B_0^*$ can be written as follows 
\begin{equation}
 B^*\= 
 \begin{bmatrix}
  \dfrac{1}{1+ c_\alpha} & 0 \\
  0 & -\dfrac{(2-\alpha)^2}{8}\pot(s_0)- c_\alpha \overline \delta_\alpha^2  \end{bmatrix}\diamond T^*
\end{equation}
where the first endomorphism is defined on $\R s_0$, while the second on $s_0^\perp$.
With  respect to the $\diamond$-product, the matrix $H^*$ can be written as 
\begin{equation}
 H^*= 
 \begin{bmatrix}
 0 & \dfrac{(2-\alpha)^2}{8}\pot(s_0) + c_\alpha \overline \delta_\alpha^2 \\
  \dfrac{1}{1+ c_\alpha} & 0
 \end{bmatrix}\diamond JT^*
\end{equation}
We obtain that $H^*$ is symplectically similar to 
\begin{equation}
\widehat{H^*} = \begin{bmatrix}
 0 &\dfrac{(2-\alpha)^2}{8}\pot(s_0)+ c_\alpha \overline \delta_\alpha^2 \\
  \dfrac{1}{1+ c_\alpha} & 0
 \end{bmatrix}\diamond  J\widehat{T^*},
\end{equation}
indeed, by taking into account Lemma \ref{thm:tecnico2}, 
$J\widehat{T^*} = P^{-1}(J{T^*})P$ for some symplectic matrix $P$, then $\widehat{H^*} = Q^{-1}{H^*}Q$, with $Q := I \diamond P$.

The hyperbolicity of $H^*$ actually depends only on the hyperbolicity of 
\begin{equation}
 J\widehat{T^*} = \begin{bmatrix}
  0 &   \dfrac{(2-\alpha)^2}{8} \pot(s_0)\Id + D^2\widetilde U(s_0)\\
  \Id & 0 
 \end{bmatrix},
\end{equation}
in fact, we recall that both constants $c_\alpha$ and $\dfrac{(2-\alpha)^2}{8}\pot(s_0)+ c_\alpha \overline \delta_\alpha^2 = 2U(s_0)$ are positive.
We now  observe that 
\[
\big(J\widehat{ T^{*}})^2\= 
\begin{bmatrix}
 \dfrac{(2-\alpha)^2}{8} \pot(s_0) \Id+D^2 \widetilde U(s_0)&0\\
0 &  \dfrac{(2-\alpha)^2}{8} \pot(s_0) \Id+{D^2\widetilde U(s_0)}
%+D^2 \pot|_{\mathcal E}(s_0)
\end{bmatrix}.
\] 
Then the  eigenvalues of $J\widehat{T^*}$ are plus/minus the square root of the eigenvalues of 
\[
%\dfrac{(2-\alpha)^2}{8} \pot(s_0) \Id+{\color{red}D^2\widetilde U(y_0)}
%+D^2 \pot|_{\mathcal E}(s_0).
\]
Since \BS-condition holds if and only if the spectrum of the previous matrix in entirely contained in
 $(0,+\infty)$, we claim that $J\widehat{T^*}$, hence $\widehat{H^*}$, is hyperbolic if and only if \BS-condition holds.
\end{proof}
By taking into account  of Proposition \ref{prop:bs-iperb}, we are in position to prove the main result of this section.
\begin{thm}\label{thm:key1fredholm}
	Let $y$ be an $s_0$-a.s. and let $\qAS$ be the $s_0$-asymptotic quadratic form of $y$ introduced in Definition \ref{def:index_quadr_AS}. Then $\qAS$ is Fredholm on $W^{1,2}_0\big([0,+\infty);\X\big)$ if and only if the \BS-condition holds. 
\end{thm}
\begin{proof}
We start to consider the $s_0$-asymptotic index form of $y$, given in Definition \ref{def:index_quadr_AS}, namely
\begin{equation*}
	\IAS(u,v) =
	\int_0^{+\infty}\left[
	\Mprod{ P(\tau)u'}{v'}+ 
	\Mprod{ Q(\tau)u}{v'}+
	\Mprod{\trasp{Q}(\tau)u'}{v} + \Mprod{\widetilde { R}(\tau)u}{v}\right]\, d\tau,
\end{equation*}
hence, after an integration by part, we obtain
\begin{equation}
\label{eq:2--1}
\IAS(u,v) = \langle \mathcal T u,v \rangle_{L^2(0,+\infty)}.
\end{equation}
Let $T: W_0^{1,2}\big([0,\infty)\big)\to L^2(0,+\infty)$ be the differential operator defined by  $
Tu\=\dfrac{d\,u}{d\tau}$ 
and let us consider the (unbounded) self-adjoint operator $\mathcal C$ (in $L^2$ with dense domain $\mathcal D(\mathcal C)= W^{2,2}\big([0,\infty) \big)$ defined by $\mathcal C\= (\Id + T^*T)^{-1}$. Now,  we observe that the operator $I+ T^*T$ is clearly a selfadjoint (in $L^2$) Fredholm operator, since it is an equivalent norm in $W_0^{1,2}\big([0,+\infty)\big)$; in fact,
\begin{eqnarray}\label{eq:00}
\langle(I+T^*T)u, v\rangle_{L^2(0,+\infty)}
= \int_0^{+\infty} \left[\Mprod{u}{v}- \Mprod{\dfrac{d^2u}{d\tau^2}}{v} \right] d\tau \\
= \int_0^{+\infty} \left[\Mprod{u}{v}+ \Mprod{\dfrac{du}{d\tau}}{\dfrac{dv}{d\tau}} \right]d\tau.
\end{eqnarray}
By this fact, readily follows that $\mathcal C$ is a self-adjoint (in $L^2$) Fredholm operator and we observe that the form $\IAS$  can be written as follows
\begin{equation}
\IAS(u,v) =  \langle \mathcal C\,\mathcal T u,v \rangle_{W_0^{1,2}(0,+\infty)}, \qquad \forall \, u,v \in W^{1,2}_0([0,+\infty);\X).
\end{equation}
In order to conclude the proof, it is enough to prove that $ \mathcal C\,\mathcal T$ is a Fredholm operator if and only if \BS-condition holds. Being $\mathcal C$ a Fredholm operator we need to show that $\mathcal T$ is a Fredholm operator if and only if \BS-condition holds. 
By invoking \cite[Theorem 4.1]{RS05a}, we get that $\mathcal T$ is a Fredholm operator if and only if the first order Hamiltonian system associated to the  $s_0$-limit Sturm-Liouville operator $\mathcal T^* $ is hyperbolic. We conclude the proof using Proposition \ref{prop:bs-iperb}. 
\end{proof}
Summing up, Theorem \ref{prop:key1-una freccia} and Theorem \ref{thm:key1fredholm}, we obtain the following result that will allow the construction of the index theory of Section \ref{sec:main}.
Let us recall that the form $\qAS$ associated to an $s_0$-a.s. $y$ is termed \emph{essentially positive} if $\ispec(y) < +\infty$ (cf. Proposition \ref{thm:as69}).
\begin{cor}\label{cor:tuttoinsieme}
Let $y$ be an $s_0$-a.s. and let $\qAS$ be the $s_0$-asymptotic quadratic form of $y$ introduced in Definition \ref{def:index_quadr_AS}. Then $\qAS$ is 
(bounded) essentially positive and Fredholm on $W^{1,2}_0\big([0,+\infty)\big)$ if and 
only if the \BS-condition holds. 
\end{cor}
For any  $\tau_0 \in [0,+\infty)$, let $
\psi_{\tau_0}:[0,+\infty)\to \Sp(2N)$ be the 
matrix-valued solution of the Hamiltonian system given in Equation 
\eqref{eq:Ham-L2-coll} such that $\psi_{\tau_0}(\tau_0)=\Id$. 
We then define  the {\em stable space\/} as 
\begin{equation}\label{eq:stabiledellaflia}
E^s(\tau_0)=\Set{v \in \R^{2N}| \lim_{\tau \to 
		+\infty}\psi_{\tau_0}(\tau)\, v=0}
\end{equation}
and we set 
\begin{equation}\label{eq:stabile-infinito}
E_*^s \= \Set{v \in \R^{2N}| \lim_{\tau \to +\infty}e^{\tau H^*}\, v=0},
\end{equation}
the stable space of an autonomous system  $z'(\tau)= H^*\, z(\tau)$, which it doesn't depend on 
the initial condition $\tau_0$. 
In order to define the geometrical index as an intersection index of a smooth path of Lagrangian subspace the next result will be crucial. At first, let us recall that the Grassmannian $G_N(\R^{2N})$ is a smooth $N^2$-dimensional manifold consisting of all $N$-dimensional subspaces of $\R^{2N}$ and the Lagrangian Grassmannian is an analytic submanifold of the Grassmannian manifold. Its topology can be described by the metric 
\begin{equation}\label{eq:gap-metric}
d(V, W)\= \norm{P_V- P_W},
\end{equation}
where $P_V$ and $P_W$ denote the orthogonal projections in $\R^{2N}$ onto the subspaces $V$ and $W$, respectively. We refer to the metric defined in Equation \eqref{eq:gap-metric} as the gap  metric.  
\begin{lem}\label{thm:stabili-Lagrangiani}
	We assume that the \BS-condition holds. For every $\tau_0\in [0,+\infty)$, 
	the subspace $E^s(\tau_0)$ defined in Equation 
	\eqref{eq:stabiledellaflia} and $ E^s_*$ defined in Equation \eqref{eq:stabile-infinito} belong to the Lagrangian 
	Grassmannian manifold  $\Lagr(N)$ of  $(\R^{2N}, \omega)$. 
	Moreover the path $\tau_0 \mapsto E^s(\tau_0)$ is smooth. 
\end{lem}
\begin{proof}
    We start to recall that the dimension of an isotropic subspace in $\R^{2N}$ is at most $N$ and an isotropic subspace is Lagrangian if and only if its dimension is precisely $N$.  
    
    Let $v, w : [0,+\infty) \to \R^{2N}$ be solutions of the Hamiltonian system given in Equation \eqref{eq:Ham-L2-coll}. Then $\omega\big(v(\tau), w(\tau)\big)$ is constant for all $\tau \in [0,+\infty)$. In fact, by a direct computation, we get
   {\begin{multline}
    \dfrac{d}{d\tau}\Big(\omega\big(v(\tau), w(\tau)\big)\Big)= \dfrac{d}{d\tau}\Big(\langle J\,v(\tau), w(\tau)\rangle\Big)= \langle Jv'(\tau), w(\tau) \rangle + \langle J v(\tau), w'(\tau)\rangle\\= 
    \langle B(\tau)\,v(\tau), w(\tau) \rangle + \langle  v(\tau), \trasp{J}w'(\tau)\rangle=\langle B(\tau)\,v(\tau), w(\tau) \rangle -\langle  v(\tau), B(\tau)\, w(\tau)\rangle=0
    \end{multline}}
    where the last equality follows being $B(\tau)$ symmetric for every $\tau \in [0,+\infty)$. 
    
    Now, if $v(\tau_0), w(\tau_0) \in E^s(\tau_0)$ for some $\tau_0 \in [0,+\infty)$, this clearly implies that $\omega\big(v(\tau), w(\tau)\big)=0$. Since the same arguments applies to the Equation \eqref{eq:Ham-L2-coll-limit-nuovo}, we conclude that the spaces $E^s(\tau_0), E^s_*$ are isotropic. Let us observe that $E^s_*$ is the generalized eigenspaces of $H^*$ with respect to eigenvalues having negative real part. Since $H^*$ is hyperbolic (cf. Proposition \ref{prop:bs-iperb}), we conclude that $E^s_* \oplus E^u_*=\R^{2N}$. Thus, these generalized eigenspaces are of dimension $N$ and being isotropic, they are  Lagrangian. 
    
	Now, by invoking \cite[Proposition 1.2]{AM03}, we obtain 
	that 
	\begin{equation}\label{eq:conv}
	\lim_{\tau_0 \to +\infty} E^s(\tau_0)= E^s_*
	\end{equation}
	where the limit is meant in the gap topology (i.e. the metric topology induced 
	by the gap metric, defined above) of the Lagrangian Grassmannian. As direct consequence of the convergence result stated in Equation 
	\eqref{eq:conv} and by taking into account that $E^s_*$ is a Lagrangian subspace {and that the dimension is a locally constant function (being a discrete-valued  continuous function on the Lagrangian Grassmannian),} it readily follows that $E^s(\tau) \in \Lagr(N)$ for every $\tau \in [0,+\infty)$.
	
	The second claim follows by \cite[Theorem 3.1]{AM03}.
	For further details, we refer the  interested reader to 
	\cite{HP17} and references therein. This conclude the  proof. 
\end{proof}

We are now entitled to define the geometrical 
index. 
\begin{defn}\label{def:geo-index}
	Let $y$ be a $s_0$-a.s. and  $L_0\= \R^N \times \{0\} \subset \R^N \times \R^N$ be 
	the (horizontal) Dirichlet Lagrangian. We define the {\em geometrical index\/} of $y$ as the integer given by 
	\[
	\igeo(y)\= -\iCLM\big(E^s(\tau_0), L_0; \tau_0 \in [0,+\infty)\big),
	\]
	where the integer $\iCLM$ is the Maslov index for ordered pairs of paths of 
	Lagrangian subspaces.
\end{defn}
\begin{defn}\label{def:BND}
	The $s_0$-homothetic Hamiltonian system ($s_0$-asymptotic Hamiltonian 
	system) given in  Equation \eqref{eq:H-coll-1-limit-nuovo} (resp. in Equation \eqref{eq:Ham-L2-coll})
	is termed to satisfy the  {\em boundary non-degenerate condition, BND-condition  for short\/}
	if 
	\begin{equation}\label{eq:hyperbolicity-analogous-L}
	L_0 \cap E^s_*=\{0\}\quad(\textrm{resp. } L_0 \cap E^s(0)=\{0\} \textrm{ 
		and } L_0 \cap E^s_*=\{0\}).
	\end{equation}
\end{defn}

\begin{lem}\label{thm:BND1}
 If the \BS-condition holds, then the $s_0$-limit Hamiltonian system given in 
 Equation \eqref{eq:Ham-L2-coll-limit-nuovo} is BND. 
\end{lem}
\begin{proof}
In order to prove this result, we start to consider the $2 \times 2$ Hamiltonian matrix
\begin{equation}
K=
\begin{bmatrix}
 0& r\\
 p&0
\end{bmatrix}, \qquad p, r >0.
\end{equation}
By a direct computation, the eingenvalues are $\lambda_\pm\=\pm\sqrt{pr}$
whose corresponding eigenvectors 
\[
e_\pm\= \begin{bmatrix}
       1\\
\pm \sqrt{p/r}      
\end{bmatrix}.
\]
By this, it follows that the stable space is generated by the eigenvector $e_-$ and since $p \neq0$ this 1D-
subspace is always transversal to $L_0 = (1,0)$. 
Since the 
restriction of $R^*$ to $T_{s_0}\mathcal E$ reduced to $\dfrac{(2-\alpha)^2}{8} \pot(s_0)\,I$, in a suitable basis, we can write
\begin{multline}
  \begin{bmatrix}
  \Id & -c_\alpha \overline \delta_\alpha\Id \\
  -c_\alpha \overline \delta_\alpha\Id & ( c_\alpha \overline \delta_\alpha)^2\Id- \widetilde R_0^*
 \end{bmatrix}\\ = \begin{bmatrix}
  1 & -c_\alpha \overline \delta_\alpha \\
  -c_\alpha \overline \delta_\alpha & (c_\alpha \overline \delta_\alpha)^2 - 
  \dfrac{(2-\alpha)^2}{8} \pot(s_0)- \mu_1
 \end{bmatrix}\diamond \cdots \diamond
 \begin{bmatrix}
  1 & -c_\alpha \overline \delta_\alpha \\
  -c_\alpha \overline \delta_\alpha & (c_\alpha \overline \delta_\alpha)^2 - 
  \dfrac{(2-\alpha)^2}{8} \pot(s_0)-\mu_N
 \end{bmatrix}
\end{multline}
where for $j=1, \dots, N$, $\mu_j$ are the eigenvalues of $D^2 \pot|_{\mathcal E}(s_0)$. 
We set
\begin{multline}
 S^*_0\= \begin{bmatrix}
  \dfrac{1}{1+ c_\alpha} & 0 \\
  0 &-\left[ \dfrac{(2-\alpha)^2}{8}\pot(s_0)+ c_\alpha \overline \delta_\alpha^2\right]
 \end{bmatrix} \textrm{ and }\\
 S^*_j\=  \begin{bmatrix}
  1 & -c_\alpha \overline \delta_\alpha \\
  -c_\alpha \overline \delta_\alpha & (c_\alpha \overline \delta_\alpha)^2 - 
  \dfrac{(2-\alpha)^2}{8} \pot(s_0)-\mu_j
 \end{bmatrix}
\end{multline}
and we observe that 
\[
H^*= K_0^* \diamond \cdots \diamond K_N^*
\]
where $K_j^*\= J S_j^*$, $j=0, \dots, N$.  
By Lemma \ref{thm:tecnico} applied to each $K_j^*$, we 
already know that it is symplectically similar to 
\[
  \widehat K_0^*= 
   \begin{bmatrix}
    0 &  \dfrac{(2-\alpha)^2}{8}\pot(s_0)+ c_\alpha \overline \delta_\alpha^2\\
    \dfrac{1}{1+ c_\alpha} & 0 
   \end{bmatrix},
  \widehat K_j^*= 
   \begin{bmatrix}
    0 & \dfrac{(2-\alpha)^2}{8} \pot(s_0)+ \mu_j\\
    1 & 0 
   \end{bmatrix} \text{ for }j=1,\ldots, n. 
\]
By applying the previous analysis carried out in the $2 \times 2$ case, we obtain that 
$E^s_* \cap L_0 =\{0\}$  which is the same that $s_0$-limit Hamiltonian system given in Equation \eqref{eq:Ham-L2-coll-limit-nuovo} is BND. This conclude the proof. 
\end{proof}
\begin{rem}
The Definition \ref{def:BND} is a  transversality condition between two 
Lagrangian subspaces  which has the same flavour of  the asymptotic 
hyperbolicity  condition (transversality between the stable and unstable subspaces) 
required in developing an index theory for homoclinic
solutions of  Hamiltonian systems.  
\end{rem}

%========================================
%========================================
\section{The Index Theorem}\label{sec:main} 
%========================================
%========================================

This section  is devoted to the proof of Morse type index Theorem  
for $s_0$-asymptotic motions by relating in  a precise way the 
spectral index to the geometrical index introduced in Section \ref{sec:fredholmness}.  
For the sake of the reader the proof will be scattered along the whole section. A key result of the section is Lemma \ref{lem:utile}, where we prove that the spectral index is the  spectral flow of a (continuous) path of bounded Fredholm quadratic forms. 

\begin{thm}\label{thm:indextheorem} 
({\bf Index theorem for a $s_0$-a.s.}) 
Let $y$ be  a $s_0$-a.s. and we assume that  
the \BS-condition holds.
Then we have
\begin{equation}\label{eq:equality}
  \ispec(y)=\igeo(y).
\end{equation}
\begin{figure}[ht]
\centering
\includegraphics[scale=0.25]{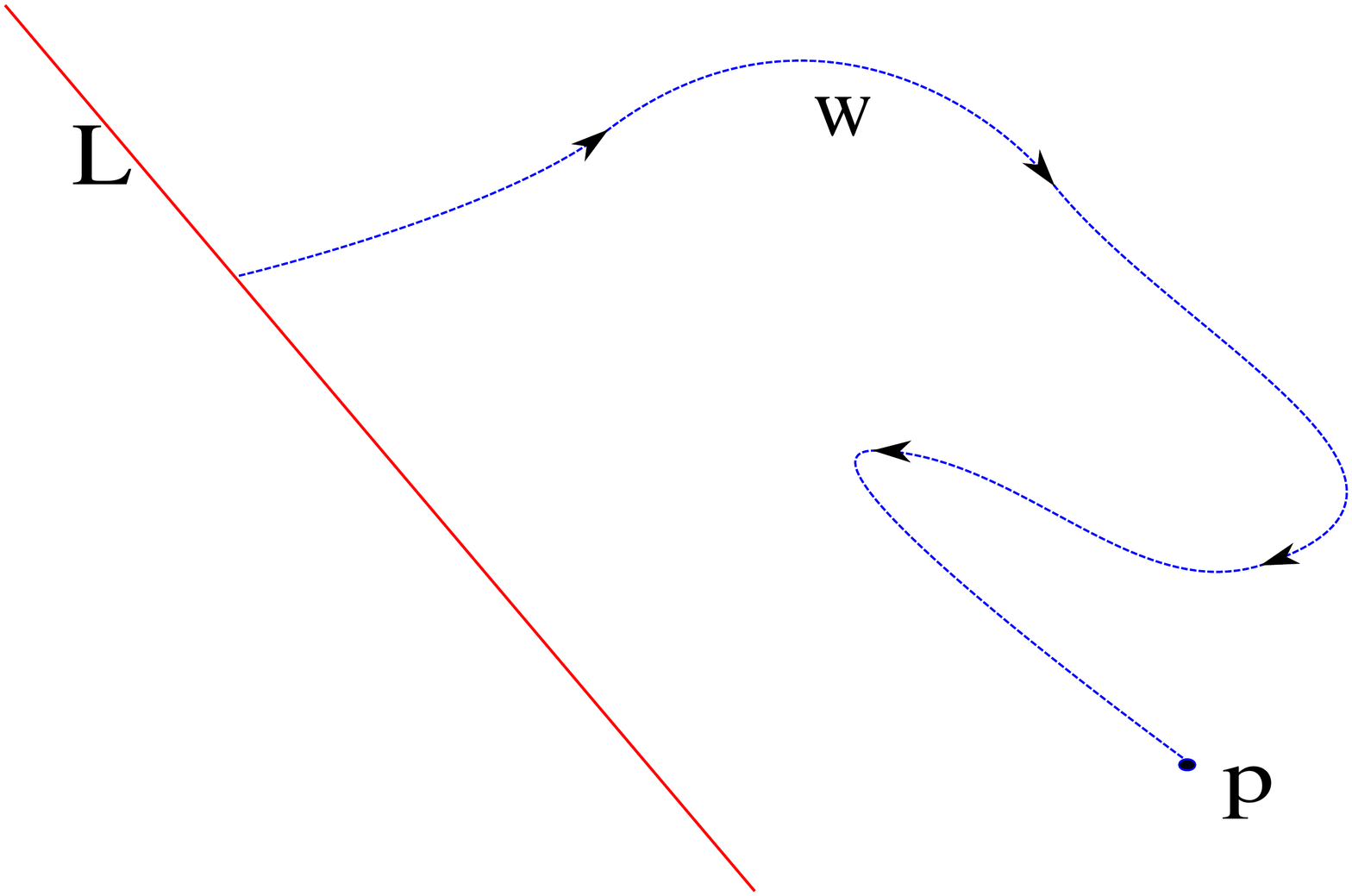}
\caption{A future halfclinic orbit (dashed blue line) between $L$ and the point 
$p$.}\label{fig:halfline}
\end{figure}
\end{thm}
\begin{rem}
The equality given in Equation \eqref{eq:equality} allow us to mirror the problem of 
computing the Morse index of a $s_0$-a.s. (integer associated to an unbounded Fredholm operator
in an {\em infinite dimensional\/} separable Hilbert space) into an 
intersection index between a curve of Lagrangian 
subspaces and a {\em finite dimensional\/} transversally oriented variety. 
\end{rem}
\begin{rem}
The idea behind the proof of this result  relies on the fact that the spectral index can be 
related to another topological invariant known in literature as spectral flow. 
Moreover under suitable transversality assumptions, 
the local contribution to the spectral flow as well as to the Maslov index are given by the sum of the signature of quadratic forms on a finite dimensional 
vector spaces. 
In order to conclude the proof it is enough to prove that these local contributions coincide. 
\end{rem}
For  $\sigma \in [0,+\infty)$,  we consider the quadratic form defined by
\begin{equation}\label{eq:quadratic-sigma}
\qAS_\sigma(u)\= \qAS(u) + \sigma \norm{u}^2_{W_0^{1,2}}
\end{equation}
where $\qAS$ is the quadratic form associated to the index form defined in Equation \eqref{eq:IAS}.
By arguing precisely as in Theorem \ref{thm:key1fredholm} the following 
result holds. 
\begin{lem}\label{thm:famiglia-sono fredholm}
For each $\sigma \in [0,+\infty)$,  the quadratic form $\mathcal Q_\sigma$ given in Equation \eqref{eq:quadratic-sigma} is 
Fredholm and essentially positive if and only if the \BS-condition holds. 
\end{lem}
Integrating by parts in Equation \eqref{eq:quadratic-sigma} we obtain 
the  second order selfadjoint Fredholm  operator 
\begin{equation}\label{eq:operator-sigma}
\mathcal T_\sigma\= \mathcal T + \sigma \Id, \qquad \sigma \in [0,+\infty).
\end{equation}
Explicitly, we have
\begin{equation}\label{eq:fliasturm}
[\mathcal T_\sigma u](\tau)\=\dfrac{d}{d\tau}\Big(P(\tau)u'+ Q(\tau)u\Big) + \trasp{Q}(\tau)+ 
\widetilde R_\sigma(\tau) u, \qquad \tau \in [0,+\infty)
\end{equation}
where   $\widetilde R_\sigma(\tau)\= \widetilde R(\tau)+\sigma \Id_N$.
We define the symmetric matrix $D_\sigma(\tau)$ as 
\[
 D_\sigma(\tau)\= 
 \begin{bmatrix}
  P(\tau)& Q(\tau)\\
  \trasp{Q}(\tau) & R_\sigma(\tau)
 \end{bmatrix}, \qquad \tau \in [0,+\infty).
\]
\begin{lem}\label{lem:nuovo}
$D_\sigma(\tau)$ is positive definite 
if and only if $\widetilde R_\sigma(\tau) - \trasp{Q}(\tau) P^{-1}(\tau) Q(\tau) $ is. 
\end{lem}
\begin{proof}
In fact, by a standard result in convex analysis $D_\sigma(\tau)$ is positive definite if and only if both 
$P(\tau)$ as well as  the Schur complement of $P(\tau)$ in $D_\sigma(\tau)$ which is represented by the matrix 
 $\widetilde R_\sigma(\tau) - \trasp{Q}(\tau) P^{-1}(\tau) Q(\tau) $ is positive definite. 
(For further details, we refer the interested reader to \cite[pag.650-651]{BV04} and 
references therein). Now the thesis readily follows by invoking Lemma \ref{lem:spectra}. 
This conclude the proof. 
\end{proof}
\begin{rem}\label{rem:nuova}
As a direct consequence of Lemma \ref{lem:nuovo} and Lemma \ref{lem:spectra}, there exists 
$\sigma_0>0$ sufficiently large such that 
\begin{equation}\label{eq:sigmazero}
D_{\sigma_0}(\tau)>0 \qquad \forall\, \tau \in [0,+\infty)
\end{equation}
and by this we infer that $ \qAS_{\sigma_0}$ is non-degenerate. 
\end{rem}
\begin{lem}\label{lem:utile}
Let $y$ be a $s_0$-a.s. If the \BS-condition holds then  we obtain 
\begin{equation}\label{eq:relazione-morse-flia}
\ispec(y)= \spfl(\qAS_\sigma; \sigma \in [0,\sigma_0]),
\end{equation}
where $\spfl$ denotes the spectral flow (cf. Definition \ref{def:sfquadratic}). 
\end{lem}
\begin{proof}
By\cite[Formula 2.7]{ZL99} and Remark \ref{rem:nuova}, we get that 
\[
 \spfl(\mathcal Q_\sigma; \sigma \in [0, \sigma_0]) = \dim E_-(\mathcal Q_0) = \dim E_-(\mathcal Q)
\]
and by taking into account  Definition \eqref{def:spectral-index}, the thesis readily follows.  
\end{proof}
By setting $z(\tau)=\big(P(\tau)u'(\tau)+ Q(\tau)u(\tau), u(\tau)\big)$ we obtain that $u $ is a 
solution of Equation \eqref{eq:fliasturm} if and only if $z$  solves the following linear Hamiltonian system
\begin{equation}\label{eq:family-ham-sys}
z'(\tau)= J  B_\sigma(\tau)\, z(\tau) , \qquad \sigma \in[0,\sigma_0] 
\end{equation}
for
\[
B_\sigma(\tau) \= 
 \begin{bmatrix}
  P^{-1}(\tau) &
  -P^{-1}(\tau)Q(\tau)\\
  -\trasp{Q}(\tau)P^{-1}(\tau) & 
\trasp{Q}(\tau)P^{-1}(\tau)Q(\tau)-\widetilde R_\sigma(\tau) 
 \end{bmatrix}.
\]
For any $(\tau_0, \sigma) \in [0,+\infty)\times [0,\sigma_0]$, let $\psi_{(\tau_0, \sigma)}: [0,+\infty) \to \Sp(2N)$ 
be the matrix-valued solution of the system given in 
Equation \eqref{eq:family-ham-sys} such that $\psi_{(\tau_0,\sigma)}(\tau_0)=\Id$. As before, we then define the 
{\em stable space\/} of the Hamiltonian system \eqref{eq:family-ham-sys} as follows
\begin{equation}\label{eq:stabile-sflia}
 E_\sigma^s(\tau_0)= \Set{v \in \R^{2N}| \lim_{\tau\to+\infty}\psi_{(\tau_0,\sigma)}(\tau)\, v=0}.
\end{equation}
Repeat verbatim  the arguments given in the proof of Lemma \ref{thm:stabili-Lagrangiani}, the following result holds (cf. Appendix \ref{sec:spectral-flow-and-Maslov}  for further details). 
\begin{lem}\label{thm:stabili-Lagrangiani-flia}
We assume that the \BS-condition holds. For every $(\tau_0, \sigma)\in [0,+\infty)\times [0,\sigma_0]$, 
the subspace $E^s_\sigma(\tau_0)$ defined in Equation \eqref{eq:stabile-sflia} belongs to the Lagrangian Grassmannian manifold $\Lagr(2N)$. Moreover, for every $\tau_0 \in [0,+\infty)$, the path $\sigma \mapsto E_\sigma^s(\tau_0)$ is analytic.
\end{lem}
For every $\tau_0 \in [0,+\infty)$, we can then associate to the family of Hamiltonian systems given in Equation \eqref{eq:family-ham-sys}, the (analytic) path in $\Lagr(2N)$, $\sigma \mapsto  E^s_\sigma(\tau_0)$, 
as well as the Maslov index 
\begin{equation}\label{eq:Maslov-sigma}
 \iCLM(E^s_\sigma(\tau_0), L_0;\sigma \in[0,\sigma_0]).
\end{equation}
We consider on the Sobolev space 
\[
W^{1,2}_{L_0}([0,+\infty), \R^{2N})\=\Set{u\in W^{1,2}([0,+\infty), \R^{2N})| u(0) \in L_0}
\]
the path $\sigma \mapsto \mathcal A_\sigma$  
of  closed operators, selfadjoint in $L^2([0,+\infty),\R^{2N})$ defined as follows 
\[
 \mathcal A_\sigma\=- J \dfrac{d}{d\tau}- B_\sigma (\tau): W^{1,2}_{L_0}([0,+\infty), 
\R^{2N})\subset  L^2([0,+\infty), \R^{2N})\to L^2([0,+\infty), \R^{2N}).
\]
\begin{rem}\label{rmk:utile}
 It is worth noticing that 
 \begin{equation}
  -\spfl\big(\mathcal A_\sigma; \sigma \in [0,\sigma_0]\big)=  
  \irel\left( -J \dfrac{d}{dt}-  B(t), - J \dfrac{d}{dt}- B_{\sigma_0}(t) \right)
 \end{equation}
as a direct consequence of \cite[Definition  2.8]{ZL99} (cf. Definition \ref{def:relativeMorseindex}). 
\end{rem}
The next technical result, is crucial in order to relate the geometric index to the Maslov introduced in Equation \eqref{eq:Maslov-sigma}.
\begin{lem}\label{thm:rinuovo}
Let us consider the operator $\mathcal  T_\sigma$ introduced in Eq. \eqref{eq:fliasturm}.
Then $u \in \ker \mathcal T_\sigma$  
if and only if  
\[
z(\tau)=\big(P(\tau) u'(\tau) + Q(\tau)u(\tau), u(\tau)\big)
\]
belongs to 
$\ker \mathcal A_\sigma$. Furthermore the Hamiltonian boundary value problem 
 \begin{equation}
  \begin{cases} 
   z'(\tau) = JB_\sigma(\tau) z(\tau), \qquad \tau \in [0,+\infty)\\
   z(0) \in L_0, \qquad  \lim_{\tau \to +\infty} z(\tau)=0
  \end{cases}
 \end{equation}
admits only the trivial solution if and only if $\mathcal A_\sigma$ is nondegenerate. 
\end{lem}
\begin{proof}
The first statement follows immediately by a straightforward calculation. 
In order to prove the second statement, it is enough  to observe that for each $\sigma \in [0,\sigma_0]$ the 
evaluation map
\begin{equation}\label{eq:iso-ev}
 \ker \mathcal T_\sigma \ni u \longmapsto u(0) \in  L_0 \cap E_\sigma^s(0)
\end{equation}
is an isomorphism.
\end{proof}
We denote by $E^s_{*, \sigma}$ the stable space of the one-parameter family of Hamiltonian systems pointwise 
defined by 
\begin{equation}\label{eq:flia-ham-bnd}
 z'(\tau)= J B_\sigma^* z(\tau), \tau \in [0,+\infty)
\end{equation}
where 
\[
 B_\sigma^*\=  
 \begin{bmatrix}
  P_0^{-1} &
  -P_0^{-1}Q_0\\
  -\trasp{Q}_0P^{-1}_0  & 
\trasp{Q}_0P^{-1}_0Q_0-\widetilde R_\sigma^* 
 \end{bmatrix}.
\]
for $\widetilde R_\sigma^* \= \widetilde R_0^* + \sigma \Id$. 
\begin{lem}\label{thm:BND2}
  If the \BS-condition holds, then the $s_0$-limit family of Hamiltonian system given in 
 Equation \eqref{eq:flia-ham-bnd} is BND. 
\end{lem}
\begin{proof}
 The proof of this result readily follows by arguing as in the proof of Lemma \ref{thm:BND1}. 
\end{proof}
The next result (which is crucial in the proof of Theorem \ref{thm:indextheorem}), relates the geometrical index of a $s_0$-a.s. defined as (Maslov) intersection index between a curve of Lagrangian subspaces parametrized by the $\tau$ variable with the Dirichlet Lagrangian and  the  intersection index of another completely different  curve of Lagrangian subspaces parametrized by the new parameter $\sigma$. The key point in order to establish such an equality is the based on the fact that the rectangle is homotopically  trivial (being contractible).
\begin{prop}\label{thm:key4}
Under the previous notation and assumptions of Lemma \ref{thm:stabili-Lagrangiani-flia}, we obtain  
\[
-\igeo(y) = \iCLM( E^s_\sigma(0),L_0;\sigma \in[0,\sigma_0]).
\]
\end{prop}
\begin{proof}
We start  to observe that as a direct consequence of the asymptotic estimates obtained in 
Section \ref{sec:variational_setting} and by invoking \cite[Proposition 1.2]{AM03} we obtain 
that 
\[
 \lim_{\tau_0\to+\infty} E^s_\sigma(\tau_0)=E^s_{*,\sigma}
\]
in the metric topology of the Lagrangian Grassmannian. 
We now consider (infinite) rectangle $\mathcal R$ 
obtained by compactifying the strip $[0,+\infty) \times [0,\sigma_0]$. Being $\mathcal R$ homotopically trivial 
(being contractible), in particular we obtain that the Maslov index with respect to $L_0$ of the part obtained 
by restricting the two-parameter family $(\tau_0, \sigma)\mapsto E^s_\sigma(\tau_0)$ to the boundary 
$\partial \mathcal R$ of the rectangle $\mathcal R$ is identically zero.
By the invariance of the $\iCLM$ index  for fixed ends homotopy and the 
additivity for concatenation of paths, we obtain that
\begin{figure}[ht]
  \centering
\includegraphics[scale=0.25]{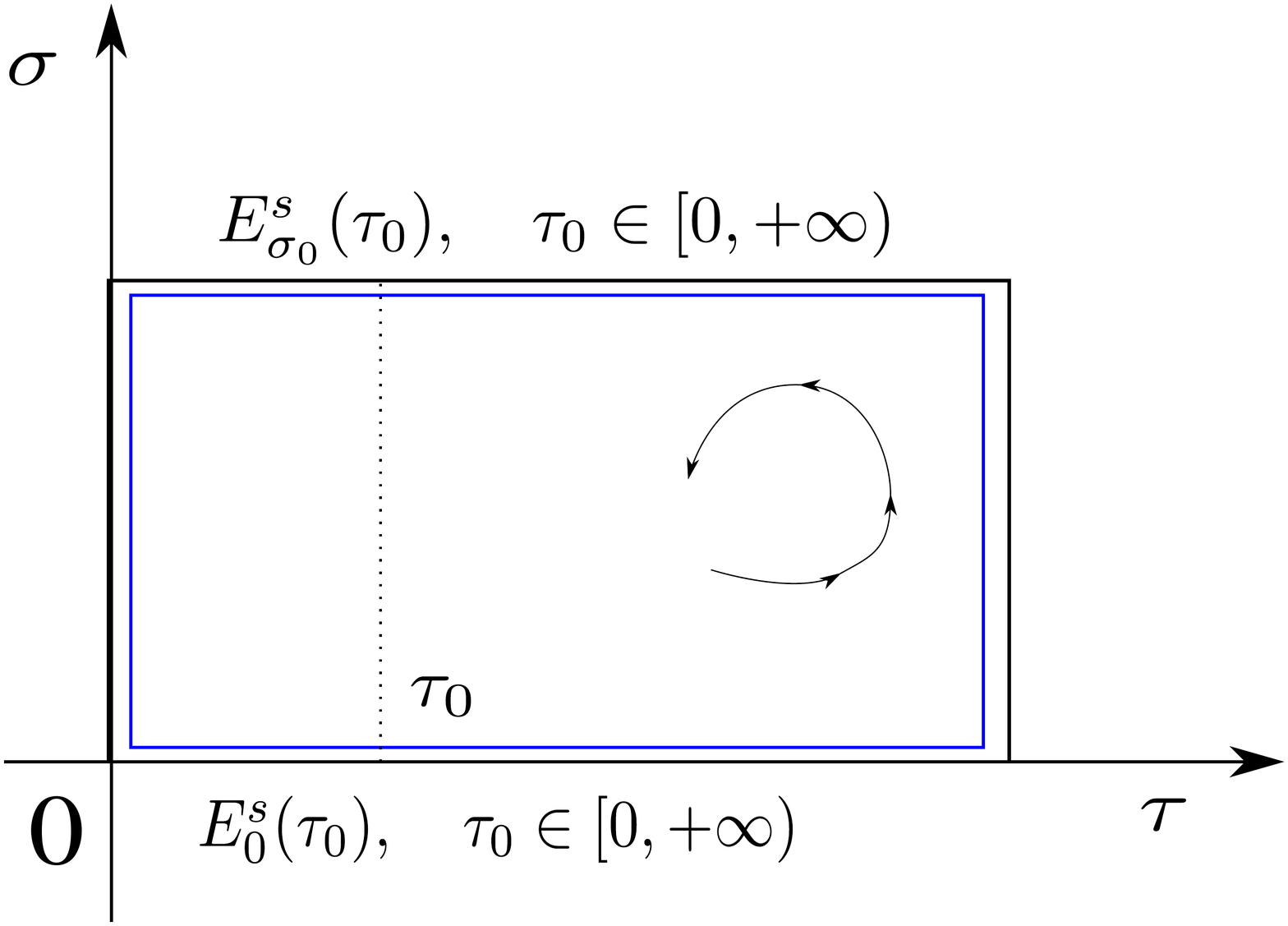}
  \caption{Pictorial view of the stable (blue rectangle)  
space.}\label{fig:stable}
\end{figure} 
\begin{multline}\label{eq:mancava}
\iCLM\big( E^s_\sigma(0),L_0;\sigma \in [0,\sigma_0])=
\iCLM\big(E^s_0(\tau_0), L_0 ;\tau_0 \in [0,+\infty)\big)\\
+ \mu\big(E^s_{*,\sigma}, L_0; \sigma \in [0,\sigma_0]\big) - \iCLM\big(E^s_{\sigma_0}(\tau_0),L_0;\tau_0 \in [0,+\infty)\big);
\end{multline}
we observe that $\iCLM\big(E^s_{*,\sigma}, L_0 ;\sigma \in [0,\sigma_0]\big)$ 
vanishes since, by  taking into account Lemma \ref{thm:BND2}, under the the \BS-condition  $E^s_{*,\sigma}$ is always transverse to $ L_0$. Furthermore the term $\iCLM\big(E^s_{\sigma_0}(\tau_0),L_0;\tau_0 \in [0,+\infty)\big)$ vanishes by Lemma \ref{lem:nuovo}, Remark \ref{rem:nuova} and Lemma \ref{lem:utile}. We conclude the proof using Definition \ref{def:geo-index}.
\end{proof}
\begin{prop}\label{thm:L=T}
We assume that the \BS-condition holds. Thus we have  
\begin{equation}\label{eq:infondonelcammindinosdue}
\spfl(\mathcal T_\sigma; \sigma \in [0,\sigma_0]) 
=\spfl(\qAS_\sigma; \sigma \in  [0,\sigma_0]).
\end{equation}
\end{prop}
\begin{proof}
We start by observing that under the the \BS-condition, the bounded symmetric 
bilinear form 
$\IAS_\sigma$ associated to 
$ \qAS_\sigma$ (see Eq. \eqref{eq:quadratic-sigma})  is coercive. 
By invoking the Lax-Milgram Theorem, there 
exists a bounded and selfadjoint operator in $W^{1,2}_{L_0}((0,+\infty), \R^N)$, namely 
$\mathcal L_\sigma: W^{1,2}_{L_0}(\R^+, \R^{N})\to 
W^{1,2}_{L_0}(\R^+, \R^{N})$ 
that represents the form 
$\IAS_\sigma$; thus we have
\[
 \IAS_\sigma[\xi,\eta]= \langle \mathcal L_\sigma 
\xi,\eta\rangle_{W_0^{1,2}((0,+\infty), \R^N)} 
\qquad 
  \forall\, \xi,  \,   \eta \in W^{1,2}_0((0,+\infty), \R^N).
\]
By taking into account Lemma \ref{thm:famiglia-sono fredholm}, the operator 
$\mathcal L_\sigma$ is Fredholm too. Thus, we have 
 By Definition, we have
\begin{equation}\label{eq:50}
 \langle\mathcal L_\sigma \xi, \eta\rangle_{W^{1,2}_{L_0}}=
  \langle \mathcal T_\sigma\, \xi, \eta\rangle_{L^2}, \quad  
  \forall\, \xi \in \big(W^{1,2}_0\cap W^{2,2}\big), \ 
\forall\,\eta \in W^{1,2}_0
\end{equation}
and by this it clearly follows that $\ker \mathcal L_\sigma= \ker 
\mathcal T_\sigma$. 
Let $\mathcal B:W^{1,2}_0(\R^+, \R^N) \to W^{1,2}_0(\R^+, \R^N)$
the (unique) selfadjoint invertible  operator  such that 
$\langle u,v\rangle_{L^2}= \langle \mathcal Bu, v\rangle_{ W^{1,2}_0}$. 
For $\delta'>0$ sufficiently 
small the selfadjoint operator $\mathcal L_\sigma^\delta\= 
\mathcal L_\sigma + \delta \mathcal B$  is
Fredholm for all $0 \leq \delta < \delta'$. Moreover if $\mathcal L$ 
has invertible ends,  then the same is true for $\mathcal L^\delta$. By the homotopy 
invariance of the spectral flow, we have 
\begin{itemize}
\item $
 \spfl(\mathcal L_\sigma; \sigma \in  [0,\sigma_0])= 
 \spfl(\mathcal L_\sigma^\delta; \sigma \in  [0,\sigma_0])$ if the ends of 
 the path $\sigma \mapsto \mathcal L_\sigma$ are non-degenerate;
\item  $
 \spfl(\mathcal L_\sigma; \sigma \in  [0,\sigma_0])-n_-\big(\Gamma(\mathcal L, 0)\big)= 
 \spfl(\mathcal L_\sigma^\delta; \sigma \in  [0,\sigma_0])$ if the initial point  is degenerate,
 \end{itemize}
here $n_-$ denotes the index of the crossing form $\Gamma(\mathcal L, 0)$ (cf. Appendix \ref{sec:spectral-flow-and-Maslov} for further details). 
By  taking into account  Equation \eqref{eq:50} as well as  definition of  $\mathcal 
B$, we obtain that 
\[
 \langle \mathcal L_\sigma^\delta\, \xi, 
\eta\rangle_{W^{1,2}_{L_0}}=
  \langle \mathcal T_\sigma^\delta\, \xi, \eta\rangle_{L^2}, \quad  
  \forall\, \xi \in \big(W^{1,2}_0 \cap W^{2,2}\big), \ 
\forall\, \eta \in W^{1,2}_0
\]
for $ \mathcal T_\sigma^\delta\= \mathcal T_\sigma+\delta \Id_{L^2}$
which shows at once that $\ker \mathcal L_\sigma^\delta=\ker 
\mathcal T_\sigma^\delta$  and  the crossing forms (cf. Appendix 
\ref{sec:spectral-flow-and-Maslov} for the definition) coincide; thus we have 
$\Gamma(\mathcal L^\delta, \sigma)= \Gamma (\mathcal T^\delta, \sigma)$, for any $\sigma 
\in [0,\sigma_0]$. 
Directly from Definition \ref{def:new-spectralflow-def}, we obtain 
\begin{multline}
 \spfl(\mathcal T_\sigma; \sigma \in  [0,\sigma_0])=
\spfl(\mathcal T_\sigma^\delta; \sigma \in  [0,\sigma_0]) = 
 \sum_{\sigma \in (0,\sigma_0]} \sgn \Gamma(\mathcal T_{\sigma}^\delta, 
\sigma)\\ =  
  \sum_{\sigma \in (0,\sigma_0]} \sgn \Gamma(\mathcal L_{\sigma}^\delta, \sigma)= 
  \spfl(\mathcal L_\sigma; \sigma \in  [0,\sigma_0])
 \end{multline}
and by the previous equality on the kernels we deduce that 
\begin{equation}\label{eq:infondo}
\spfl(\mathcal T_\sigma; \sigma \in (0,\sigma_0]) =
\spfl(\mathcal L_\sigma; \sigma \in  (0,\sigma_0]).
\end{equation}
If the initial instant is non-degenerate, we obtain the desired assertion. Otherwise, it is enough to observe that 
\begin{multline}
 \spfl(\mathcal T_\sigma; \sigma \in[0,\sigma_0])=  \spfl(\mathcal T_\sigma^\delta ; 
 \sigma \in[0,\sigma_0]) + n_-\big(\Gamma(\mathcal T, 0)\big)\\= 
 \spfl(\mathcal L_\sigma^\delta ; 
 \sigma \in [0,\sigma_0]) + n_-\big(\Gamma(\mathcal L, 0)\big)= 
  \spfl(\mathcal L_\sigma; \sigma \in[0,\sigma_0]).
\end{multline}
This conclude the proof. 
\end{proof}
The next result is interesting in its own and states that, in the case we are dealing with, 
the spectral flow associated to a 
Lagrangian system actually coincides with the spectral flow of the Hamiltonian system corresponding to its Hamiltonian counterpart. 
\begin{prop}\label{thm:T=A}
Under the \BS-condition, then we have
 \begin{equation}\label{eq:infondonelcammindinostre}
\spfl(\mathcal T_\sigma; \sigma \in [0,\sigma_0]) 
=\spfl(\mathcal A_\sigma; \sigma \in  [0,\sigma_0]).
\end{equation}
\end{prop}
\begin{proof}
By \cite[Theorem 2.6]{Wat15}, there exists  $\delta>0$ sufficiently small such that the 
path  $\mathcal T^\delta\= \mathcal T+ \delta \Id$ has only regular crossings. 
Analogously, we denote by  $\mathcal A^\delta$ the corresponding induced path. 
By arguing as in  \cite[Lemma 6.6, Chapter XVIII]{GGK90}, 
it follows that by setting $C=(0,1)$
\begin{equation}\label{eq:corr-kernel}
 \ker \mathcal T^\delta_\sigma=\Set{C\,u| u \in \ker \mathcal A^\delta_\sigma} 
\textrm{ and } 
\dim\ker\mathcal T^\delta_\sigma = 
 \dim\ker \mathcal A^\delta_\sigma.
\end{equation}
In order to conclude the proof, it is enough to show that 
\[
\spfl(\mathcal T^\delta_\sigma; \sigma\in [0,\sigma_0]) =\spfl(\mathcal 
A^\delta_\sigma; \sigma \in  [0,\sigma_0]).
\]
This can be achieved, for instance,  by showing that they have the same crossing 
instants and  isomorphic crossing forms. As direct consequence of 
Equation  \eqref{eq:corr-kernel}  the crossing instants as well as the 
multiplicity of the crossing form coincides. For, let us consider 
the quadratic Lagrangian density function defined  by 
\begin{equation}\label{eq:lagra-density}
 \widetilde L_{\sigma,\delta}(\tau, u,v)= \dfrac12\Mprod{\begin{bmatrix}
P(\tau)& Q(\tau)\\
\trasp{Q}(\tau) & \widetilde R_{\sigma,\delta}(\tau)
 \end{bmatrix}\begin{bmatrix}
 v\\
 u
 \end{bmatrix}}{\begin{bmatrix}
 v\\
 u
 \end{bmatrix}}
\end{equation}
for $ \widetilde R_{\sigma,\delta}(\tau)\=  
\widetilde R_{\sigma}(\tau)+ \delta \Id$.
By taking into account the Legendre transform and invoking  Lemma \ref{lem:nuovo}, we obtain that 
$u_\sigma \in \ker \mathcal T_\sigma$ if and only if 
$z_\sigma= \left(\partial_v \widetilde L_\sigma(\tau, u_\sigma, u_\sigma'), u_\sigma\right) 
\in \ker \mathcal A_\sigma$. 
The associated Hamiltonian function is given by 
\begin{equation}\label{eq:Ham-flia}
 H_{\sigma,\delta}(p,q)= \Mprod{p}{q'}- \widetilde L_{\sigma,\delta}(\tau, q, q')= 
  \dfrac12\Mprod{B_\sigma(\tau)\begin{bmatrix}
 p\\q
 \end{bmatrix} }{\begin{bmatrix}
p\\q
 \end{bmatrix}}.
\end{equation}
By a direct  computation we obtain 
\begin{multline}
 \Gamma(\mathcal A^\delta, \sigma)(u)= \left\langle \dfrac{\partial  
 \mathcal A^\delta_\sigma}{\partial\sigma}\, u, u 
\right\rangle_{L^2}= - \int_0^{+\infty}\Mprod{\dot B_{\sigma,\delta}(\tau)u}{u}\, d\tau
 \quad \forall \, 
u\in \ker \widetilde{\mathcal A}_\sigma\\
 \Gamma(\mathcal T^\delta, \sigma)(v)= \left\langle 
 \dfrac{\partial\mathcal T^\delta_\sigma}{\partial \sigma}\,v, 
v \right\rangle_{L^2}
 = \int_0^{+\infty}  L^\delta_\sigma(\tau, v,v')\, d\tau,  \qquad \forall\,v\in 
\ker \mathcal T^\delta_\sigma.
\end{multline}
By taking into account Equation \eqref{eq:Ham-flia} it follows that 
\begin{multline}
 \Gamma(\mathcal A^\delta, \sigma)(u)=- \int_0^{+\infty}\Mprod{\dot B_{\sigma,\delta}(\tau)u}{u}= 
 -  \int_0^{+\infty} L^\delta_\sigma(\tau, v,v')\, d\tau\\= 
 \Gamma(\mathcal T^\delta, \sigma)(v) \quad \forall \, 
u\in \ker \mathcal A^\delta_\sigma, \ \ \ \forall\,v \in 
\ker \mathcal T^\delta_\sigma.
\end{multline}
The thesis follows by invoking Definition \ref{def:new-spectralflow-def} and arguing as in the 
last part of the proof of Proposition \ref{thm:L=T}. This conclude the proof. 
\end{proof}

\begin{proof}[Proof of Theorem \ref{thm:indextheorem}]
By taking into account Corollary \ref{cor:tuttoinsieme} we know that the spectral index $\ispec(y)$ 
of the $s_0$-a.s. is finite and by Lemma \ref{lem:utile}, we know  that 
\begin{equation}\label{eq:1}
\ispec(y)= \spfl(\mathcal Q_\sigma;\sigma \in [0,\sigma_0]). 
\end{equation}
By Proposition \ref{thm:L=T} and Proposition \ref{thm:T=A}, we deduce that   
\begin{equation}\label{eq:2}
\spfl(\mathcal Q_\sigma;\sigma \in [0,\sigma_0])= \spfl(\mathcal T_\sigma;\sigma \in [0,\sigma_0])=
\spfl(\mathcal A_\sigma;\sigma \in [0,\sigma_0]). 
\end{equation}
By Equation \eqref{eq:2}, we infer that 
\begin{multline}\label{eq:ultima1}
-\spfl(\mathcal A_\sigma;\sigma \in [0,\sigma_0])=-\spfl(\mathcal T_\sigma;\sigma \in [0,\sigma_0])= 
\irel(\mathcal T_0, \mathcal T_{\sigma_0})
=\iMor(\mathcal T_{\sigma_0})-\iMor(\mathcal T_0) \\ 
=- \iMor(\mathcal T_0)= -\iMor(\mathcal T)=-\ispec(y).
\end{multline}
{By Proposition \ref{thm:key4}, we infer that 
\[
-\igeo(y) = \iCLM( E^s_\sigma(0),L_0;\sigma \in[0,\sigma_0]).
\]
By invoking  \cite[Theorem 1, pag. 8]{HP17} in the future halfclinic case (by setting $w_\lambda =y$, $L_\lambda=L_0$ and finally $\lambda=\sigma$),  we deduce also that 
\begin{equation}\label{eq:ultima2}
-\spfl(\mathcal A_\sigma;\sigma \in [0,\sigma_0])= -\igeo(y).
\end{equation}
By Equation \eqref{eq:ultima1} and Equation \eqref{eq:ultima2}, we get 
\[
-\ispec(y)=-\spfl(\mathcal A_\sigma;\sigma \in [0,\sigma_0])=-\igeo(y).
\]
This conclude the proof.} %\qed
\end{proof}

\appendix
%
%========================================
%========================================
\section{Maslov index and Spectral flow}\label{sec:spectral-flow-and-Maslov}
%========================================
%========================================

The purpose of this Section is to provide the functional analytic and 
symplectic 
preliminaries behind  {\em spectral flow\/} and the {\em 
Maslov index\/}.  In Subsection \ref{subsec:Maslovindexpath} the main 
properties 
of the intersection number of curves of Lagrangian subspaces with a 
distinguished one are collected and the {\em (relative) Maslov index\/} is 
defined. In  Subsection \ref{subsec:flussospettrale} the stage is set with a 
brief review  of the  {\em  spectral flow\/} for paths of 
bounded Fredholm operators and of Fredholm quadratic forms. 
Our basic references are \cite{Phi96, FPR99,PP05, PPT04,CLM94, GPP04, RS93,MPP05, CH07, MPP07, 
PW15a, PW15b, Wat15}
from which  we borrow some notation and definitions.

%==========================
\subsection{A recap on the Maslov Index}\label{subsec:Maslovindexpath}
%==========================

Given a $2n$-dimensional (real) symplectic space $(V,\omega)$, a {\em 
Lagrangian 
subspace\/} of $V$ is an $n$-dimensional subspace $L \subset V$ such that $L = 
L^\omega$ where $L^\omega$ denotes the {\em symplectic orthogonal\/}, i.e. the 
orthogonal 
with respect to the symplectic structure. 
We denote by $ \Lagr= \Lagr(V,\omega)$ the {\em Lagrangian Grassmannian of 
$(V,\omega)$\/}, namely the set of all Lagrangian subspaces of $(V, \omega)$
\[
\Lagr(V,\omega)\=\Set{L \subset V| L= L^{\omega}}.
\]
It is well-known that $\Lagr(V,\omega)$ is a manifold. For each $L_0 \in \Lagr$, 
let 
\[
\Lagr^k(L_0) \= \Set{L \in \Lagr(V,\omega) | \dim\big(L \cap L_0\big) =k } 
\qquad k=0,\dots,n.
\]
Each $\Lagr^k(L_0)$ is a real compact, connected submanifold of codimension 
$k(k+1)/2$. The topological closure 
of $\Lagr^1(L_0)$  is the {\em Maslov cycle\/} that can be also described as 
follows
\[
 \Sigma(L_0)\= \bigcup_{k=1}^n \Lagr^k(L_0)
\]
The top-stratum $\Lagr^1(L_0)$ is co-oriented meaning that it has a 
transverse orientation. To be 
more precise, for each $L \in \Lagr^1(L_0)$, the path of Lagrangian subspaces 
$(-\delta, \delta) \mapsto e^{tJ} L$ cross $\Lagr^1(L_0)$ transversally, and as 
$t$ increases the path point 
to the transverse direction. Thus  the Maslov cycle is two-sidedly embedded in 
$\Lagr(V,\omega)$. Based on the topological properties of the Lagrangian 
Grassmannian manifold, 
it is possible to define a fixed endpoints homotopy invariant called {\em Maslov 
index\/}.

\begin{defn}\label{def:Maslov-index}
Let $L_0 \in \Lagr(V,\omega)$ and let $\ell:[0,1] \to \Lagr(V, \omega)$ be a 
continuous path. We 
define the {\em Maslov index\/} $\iCLM$ as follows:
\[
 \iCLM(L_0, \ell(t); t \in[a,b])\= \left[e^{-\varepsilon J}\, \ell(t): 
\Sigma(L_0)\right]
\]
where the right hand-side denotes the intersection number and $0 < \varepsilon 
<<1$.
\end{defn}
For further reference we refer the interested reader to \cite{CLM94} and references therein. 
\begin{rem}
 It is worth noticing that for $\varepsilon>0$ small enough, the Lagrangian 
subspaces 
 $e^{-\varepsilon J} \ell(a)$ and $e^{-\varepsilon J} \ell(b)$ are off the 
singular cycle. 
\end{rem}
One efficient way to compute the Maslov index, was introduced by authors in 
\cite{RS93} via 
crossing forms. Let $\ell$ be a $\mathscr C^1$-curve of Lagrangian subspaces 
such that 
$\ell(0)= L$ and let $W$ be a fixed Lagrangian subspace transversal to $L$. For 
$v \in L$ and 
small enough $t$, let $w(t) \in W$ be such that $v+w(t) \in \ell(t)$.  Then the 
form 
\[
 Q(v)= \dfrac{d}{dt}\Big\vert_{t=0} \omega \big(v, w(t)\big)
\]
is independent on the choice of $W$. A {\em crossing instant\/} for $\ell$ is an 
instant $t \in [a,b]$ 
such that $\ell(t)$ intersects $W$ nontrivially. At each crossing instant, we 
define the 
crossing form as 
\[
 \Gamma\big(\ell(t), W, t \big)= Q|_{\ell(t)\cap W}.
\]
A crossing is termed {\em regular\/} if the crossing form is non-degenerate. If 
$\ell$ is regular meaning that 
it has only regular crossings, then the Maslov index is equal to 
\[
 \mu\big(W, \ell(t); t \in [a,b]\big) = \coiMor\big(\Gamma(\ell(a), W; a)\big)+ 
\sum_{a<t<b} 
 \sgn\big(\Gamma(\ell(t), W; a\big)- \iMor\big(\Gamma(\ell(b), W; b\big)
\]
where the summation runs over all crossings $t \in (a,b)$ and $\coiMor, \iMor$ 
are the dimensions  of 
the positive and negative spectral spaces, respectively and $\sgn\= 
\coiMor-\iMor$ is the  signature. 
(We refer the interested reader to \cite{LZ00} and \cite[Equation (2.15)]{HS09}). 
We close this section by 
recalling some useful 
properties of the Maslov index. \\
\begin{itemize}
\item[]{\bf Property I (Reparametrization invariance)\/}. Let $\psi:[a,b] \to 
[c,d]$ be a 
continuous and piecewise smooth function with $\psi(a)=c$ abd $\psi(b)=d$, then 
\[
 \iCLM\big(W, \ell(t)\big)= \iCLM(W, \ell(\psi(t))\big). 
\]
\item[] {\bf Property II (Homotopy invariance with respect to the ends)\/}. For 
any $s \in [0,1]$, 
let $s\mapsto \ell(s,\cdot)$ be a continuous family of Lagrangian paths 
parametrised on $[a,b]$ and 
such that $\dim\big(\ell(s,a)\cap W\big)$ and $\dim\big(\ell(s,b)\cap W\big)$ 
are constants, then 
\[
 \iCLM\big(W, \ell(0,t);t \in [a,b]\big)=\iCLM\big(W, \ell(1,t); t \in 
[a,b]\big).
\]
\item[]{\bf Property III (Path additivity)\/}. If $a<c<b$, then
\[
 \iCLM\big(W, \ell(t);t \in [a,b]\big)=\iCLM\big(W, \ell(t); t \in [a,c]\big)+
 \iCLM\big(W, \ell(t); t \in [c,b]\big) 
\]
\item[]{\bf Property IV (Symplectic invariance)\/}. Let $\Phi:[a,b] \to \Sp(2n, 
\R)$. Then 
\[
 \iCLM\big(W, \ell(t);t \in [a,b]\big)= \iCLM\big(\Phi(t)W, \Phi(t)\ell(t); t 
\in [a,b]\big).
\]
\end{itemize}

%==========================
\subsection{On the Spectral Flow}\label{subsec:flussospettrale}
%==========================
%
Let $\mathcal W, \mathcal H$ be  real separable Hilbert spaces with a dense 
and 
continuous inclusion $\mathcal W \hookrightarrow \mathcal H$.
\begin{note}
We denote by  
$\mathcal{B}(\mathcal W,\mathcal H)$ the Banach  space of all linear bounded 
operators (if $\mathcal W=\mathcal H$ we use the shorthand notation 
$\mathcal{B}(\mathcal H)$); by $\mathcal{B}^{sa}(\mathcal W, \mathcal H)$ we 
denote the set of all  bounded selfadjoint operators when regarded as operators 
on  $\mathcal H$ and finally $\mathcal{BF}^{sa}(\mathcal W, \mathcal H)$ 
denotes 
the set of all bounded selfadjoint Fredholm operators and we recall that an 
operator $T \in \mathcal{B}^{sa}(\mathcal W, \mathcal H)$ is Fredholm if and 
only if its kernel is finite dimensional and its image is closed. 
\end{note}
For $T \in\mathcal{B}(\mathcal W, \mathcal H)$ we recall that the {\em spectrum 
\/} of $T$ is $\sigma(T)\= \Set{\lambda \in \C| T-\lambda I \text{ is not 
invertible}}$ and that $\sigma(T)$ is decomposed into the {\em essential 
spectrum\/} and the {\em discrete spectrum\/} defined respectively as 
$\sigma_{ess}(T) \= \Set{\lambda \in \C| T-\lambda I \notin 
\mathcal{BF}(\mathcal W, \mathcal H)}$ and $\sigma_{disc}(T)\= \sigma(T) 
\setminus \sigma_{ess}(T)$.
It is worth noting that $\lambda \in \sigma_{disc}(T)$ if and only if it is an 
isolated point in $\sigma(T)$ and $\dim \ker (T - \lambda I)<\infty$.

Let now $T \in \mathcal{BF}^{sa}(\mathcal W,\mathcal H)$, then either $0$ is 
not 
in $\sigma(T)$ or it is in $\sigma_{disc}(T)$ (cf. \cite[Lemma 2.1]{Wat15}), 
and, as a consequence of the Spectral Decomposition Theorem (cf. \cite[Theorem 
6.17, Chapter 
III]{Kat80}), the following orthogonal decomposition holds
\[
 \mathcal W = E_-(T) \oplus \ker T \oplus E_+(T),
\]
with the property
\[
 \sigma(T) \cap(-\infty, 0)= \sigma\left(T_{E_-(T)}\right) \textrm{ and } 
 \sigma(T) \cap(0,+\infty)= \sigma\left(T_{E_+(T)}\right).
\]
\noindent
\begin{defn}\label{def:Morseindex}
Let $T \in \mathcal{BF}^{sa}(\mathcal W,\mathcal H)$. If $\dim E_-(T)<\infty$ 
(resp.  $\dim 
E_+(T)<\infty$), 
we define its {\em Morse index\/} (resp. {\em Morse co-index\/})
as the integer denoted by $\iMor(T)$  (resp. $\coiMor(T)$) and defined as:
\[
 \iMor(T) \= \dim E_-(T)\qquad \big(\textrm{resp. } \coiMor(T)\= \dim E_+(T)\big).
\]
\end{defn}
\smallskip
The space  $\mathcal{BF}^{sa}(\mathcal H)$ was intensively 
investigated by Atiyah and Singer in \cite{AS69} \footnote{% 
Actually, in this reference, only skew-adjoint Fredholm operators were 
considered, but the case of bounded selfadjoint Fredholm operators presents no 
differences.} and the following important topological characterisation can be 
deduced.
\begin{prop}\label{thm:as69} (Atiyah-Singer, \cite{AS69})
The space $\mathcal{BF}^{sa}(\mathcal H)$ consists of three connected 
components:
\begin{itemize}
\item the {\em essentially positive\/} $
 \mathcal{BF}^{sa}_+(\mathcal H)\=\Set{T \in\mathcal{BF}^{sa}(\mathcal 
H)|\sigma_{ess}(T) 
  \subset (0,+\infty)}$; 
 \item the  {\em essentially negative\/} 
  $\mathcal{BF}^{sa}_-(\mathcal H)\=\Set{T \in\mathcal{BF}^{sa}(\mathcal 
H)|\sigma_{ess}(T) 
  \subset (-\infty,0)}$;
  \item the {\em strongly indefinite\/}
$ \mathcal{BF}^{sa}_*(\mathcal H)\=\mathcal{BF}^{sa}(\mathcal 
H)\setminus(\mathcal{BF}^{sa}_+(\mathcal H)
  \cup \mathcal{BF}^{sa}_-(\mathcal H)). $
  \end{itemize}
The spaces $\mathcal{BF}_+^{sa}(\mathcal H),\mathcal{BF}_-^{sa}(\mathcal H)$ 
are  contractible (actually convex), whereas $ \mathcal{BF}^{sa}_*(\mathcal 
H)$ is topological non-trivial; more precisely, $\pi_1(\mathcal{BF}^{sa}_*(\mathcal 
H))\simeq\Z.$
\end{prop}
\begin{rem}
By the definitions of the connected components of $\mathcal{BF}^{sa}$, we 
deduce 
that a bounded linear  operator is essentially positive if and only if it is a symmetric 
compact perturbation of a (bounded) positive definite selfadjoint operator. 
Analogous observation hold for essentially negative operators. 
\end{rem}
Even if in the strongly indefinite case neither a Morse index nor a Morse 
co-index can be defined another topological invariant naturally arise. 
\begin{lem}\label{thm:calkinequiv}
Let  $S,T \in \mathcal{BF}^{sa}(\mathcal W,\mathcal H)$ be two selfadjoint 
isomorphisms such that $S-T$ is a compact (selfadjoint) operator. Then, we have:
\begin{equation}\label{eq:calkinequivalent}
\dim\big(E_-(S)\cap E_+(T)\big)<\infty \textrm { and } \dim\big(E_+(S)\cap  
E_-(T)\big)<\infty.
\end{equation}
\end{lem}
\begin{proof}
Let us set $\mathcal W_0\=E_-(S)\cap E_+(T)$. We start proving that   $\dim 
\mathcal W_0 < 
\infty$ being the proof of $\dim\big(E_+(S)\cap E_-(T)\big)<\infty$ analogous. 
For, let $P: \mathcal H\to\mathcal W_0$ be the orthogonal projector. We now 
define the operator $ P(S-T)|_{\mathcal W_0}: \mathcal W_0 \to \mathcal W_0$; 
we 
conclude showing that this operator is compact and invertible. We observe that 
it is a compact operator, since $S-T$ is a compact operator, $P$ is continuous 
and the set of compact operators is an ideal of $\mathcal B(\mathcal W, 
\mathcal 
H)$. Furthermore $P(T-S)|_{\mathcal W_0}$ is selfadjoint and invertible.
\end{proof}
\begin{defn}\label{def:relativeMorseindex}
We define the {\em relative Morse index\/} of an ordered pair of selfadjoint 
isomorphisms $S, T \in \mathcal{BF}^{sa}(\mathcal W,\mathcal H)$ such that $S-T$ 
is compact, as the integer 
\[
 \irel(S,T)\= \dim\big(E_+(S)\cap E_-(T)\big)-\dim\big(E_-(S)\cap E_+(T)\big).
\]
\end{defn}
(For further details, we refer the interested reader to \cite{ZL99} and \cite[Definition 2.1]{HS09} 
and references therein).
Moreover if $T$ is 
positive definite, then $S$ is a compact perturbation of a positive definite 
operator, so, it is essentially positive and 
\[
\irel(S,T)=-\iMor(S).
\]
We are now in position to introduce the spectral flow. 
Given a  $\mathcal C^1$-path  $L:[a,b]\to\mathcal{BF}^{sa}(\mathcal W, \mathcal 
H)$, the spectral flow of $L$ counts the net number of eigenvalues crossing 0. 
\begin{defn}\label{def:crossing}
An instant $t_0 \in (a,b)$ is called a \emph{crossing instant} (or {\em 
crossing\/} for short) if $\ker  L_{t_0} \neq \{0\}$. The \emph{crossing form} 
at a crossing $t_0$ is the quadratic form defined by 
\[
 \Gamma( L, t_0): \ker  L_{t_0} \to \R, \ \ \Gamma( L, 
t_0)[u] \=\langle 
 \dot{ L}_{t_0} u, u \rangle_{\mathcal H},
\]
where we denoted by $\dot{L}_{t_0}$ the derivative of $L$ 
with respect to the parameter $t \in [a,b]$ at the point $t_0$.
A crossing is called \emph{regular}, if $\Gamma( L, t_0)$ is 
non-degenerate. If $t_0$ is a crossing instant for $L$, we refer to 
$m(t_0)$ the dimension of $\ker  L_{t_0}$.
\end{defn}
In the case of regular curve we introduce the following Definition. 
\begin{defn}\label{def:new-spectralflow-def}
 Let  $L:[a,b]\to\mathcal{BF}^{sa}(\mathcal W, \mathcal  H)$ be a $\mathscr 
C^1$-path and 
 we assume that it has only regular crossings. Then 
 \[
\spfl(L; [a,b])= \sum_{t \in (a,b)} \sgn \Gamma(L, t)- 
\iMor\big(\Gamma(L,a)\big)
+ \coiMor\big((\Gamma(L,b)\big). 
\]
\end{defn}
\begin{rem}
 Usually the Definition of the spectral flow in literature is given for 
continuous curve of selfadjoint 
 Fredholm operators and without any crossing forms. We refer the reader to 
\cite{Phi96,RS95, Wat15} and 
 references therein. Actually Definition \ref{def:new-spectralflow-def} 
represents an efficient way for 
 computing the spectral flow; in fact it requires more regularity as well as a 
transversality assumption 
 (the regularity of each crossing instant). We observe that a $\mathscr C^1$-path 
 always exists  in the homotopy class of the original path.  
\end{rem}
\begin{rem}
 It is worth noting, as already observed by author in \cite{Wat15}, that the spectral flow can be 
 defined in the more general case of continuous 
 paths of closed unbounded selfadjoint Fredholm operators that are 
 continuous with respect to the (metric) gap-topology (cf. \cite{BLP05} and references 
 therein). However in the special case in 
 which the domain of the operators is fixed, then the closed path of unbounded 
 selfadjoint Fredholm operators can be regarded as a continuous path 
 in $\mathcal{BF}^{sa}(\mathcal W, \mathcal  H)$. Moreover  this path is also continuous 
 with respect to the aforementioned gap-metric topology.
 
 The advantage to regard the paths in  $\mathcal{BF}^{sa}(\mathcal W, \mathcal  H)$ is that the 
 theory is straightforward as in the bounded case and, clearly, it is sufficient for the applications  
 studied in the present manuscript. 
\end{rem}

Following authors in \cite{MPP05} we are in position to discuss the spectral 
flow for bounded Fredholm quadratic forms.  
Let us denote by $\mathcal Q(\mathcal H)$  the set of all bounded quadratic 
forms on $\mathcal H$. By the Riesz-Fréchét Representation Theorem, for every 
$q 
\in \mathcal Q(\mathcal H)$  there exists a unique selfadjoint operator $L_q 
\in 
\mathcal B(\mathcal H)$ called the \emph{representation of $q$ with 
respect to $\langle \cdot, \cdot \rangle_\mathcal H$} such that 
\[
 q(u)=\langle L_q\, u, u \rangle, \qquad \forall u \in \mathcal H.
\]

We call $q:\mathcal H\to \R$ a \emph{Fredholm quadratic form\/}, and we write  
$q \in \mathcal{Q_F}(\mathcal H)$, if $L_q$ is Fredholm. 
The set $\mathcal{Q_F}(\mathcal H)$ is an open subset of $\mathcal Q(\mathcal 
H)$ which is stable under perturbations by weakly continuous quadratic forms 
(cf. \cite[Appendix B]{BJP14}).
A quadratic form $q\in \mathcal{Q_F}(\mathcal H)$ is called {\em 
non-degenerate\/} if the corresponding Riesz representation $L_q$ is 
invertible. 
In view of the special role played by compact operators with respect to the 
Fredholm theory we recall that a quadratic form on $\mathcal H$ is weakly 
continuous if and only if one (and hence any equivalent) of its 
representations with respect to the Hilbert space structure of $\mathcal H$  is 
a compact operator on  $\mathcal H$ (cf. \cite[Appendix B]{BJP14}).

A Fredholm quadratic form $q$ is termed \emph{essentially positive}, and we 
write $q\in \mathcal{Q_F}^+(\mathcal H)$, provided it is a weakly continuous 
perturbation of a positive definite Fredholm form. It turns out that $q\in 
\mathcal{Q_F}^+(\mathcal H)$ if and only if its representation  with respect to 
one (and hence any equivalent) Hilbert structure on $\mathcal H$ is essentially 
positive  selfadjoint Fredholm operator; the set $\mathcal{Q_F}^+(\mathcal H)$ 
si contractible (recall Proposition \ref {thm:as69}). 
\begin{defn}\label{def:sfquadratic}
Let $q:[a,b]\rightarrow \mathcal {Q_F}(\mathcal H)$ be a continuous path. 
We define  the \emph{spectral flow of $q$}  as the  spectral 
flow associated to its path of representation
\[
\spfl(q;[a,b]) \= \spfl(L_q;[a,b]).
\]
\end{defn}
By virtue of Definition \ref{def:relativeMorseindex} if $q:[a,b] \to 
\mathcal{Q_F}^+(\mathcal H)$ it turns out that
\begin{equation} \label{eq:sf-diffMorse}
\spfl(q;[a,b]) = \iMor(q_a)-\iMor(q_b).
\end{equation}
(cf., for  instance, \cite[Definition 2.8]{ZL99} and \cite{FPR99} and references 
therein for further details).

%==============================

\vspace{0.4cm}
	\noindent
	\textsc{Vivina L.~Barutello}\\
	Dipartimento di Matematica \lq\lq G.~Peano\rq\rq\\
	Università degli Studi di Torino\\
	Via Carlo Alberto, 10\\
	10123 Torino, Italy\\
	E-mail: \email{vivina.barutello@unito.it}   
   
\vspace{0.4cm}
\noindent
\textsc{Xijun Hu}\\
Department of Mathematics\\
Shandong University\\
Jinan, Shandong, 250100\\
The People's Republic of China\\
China\\
E-mail: \email{xjhu@sdu.edu.cn}

\vspace{0.4cm}
\noindent
\textsc{Alessandro Portaluri}\\
DISAFA\\
Università degli Studi di Torino\\
Largo Paolo Braccini 2 \\
10095 Grugliasco, Torino, Italy\\
Website: \url{aportaluri.wordpress.com}\\
E-mail: \email{alessandro.portaluri@unito.it}

\vspace{0.4cm}
\noindent
\textsc{Susanna Terracini}\\
Dipartimento di Matematica ``G. Peano''\\
Università degli Studi di Torino\\
Via Carlo Alberto, 10\\
10123 Torino, Italy\\
Website: \url{https://sites.google.com/site/susannaterracini/home}\\
E-mail: \email{susanna.terracini@unito.it}

\vspace{0.4cm}
\noindent
COMPAT-ERC Website: \url{https://compaterc.wordpress.com/}\\
COMPAT-ERC Webmaster \& Webdesigner: Arch.  Annalisa Piccolo
\end{document}